\documentclass[a4paper, 11pt]{amsart}   
\usepackage{mathptmx, amssymb,amscd,latexsym, eulervm}   
\usepackage{amsmath}
\usepackage{amsthm}
\usepackage{mathdots}
\usepackage[colorlinks=true,linkcolor=blue,urlcolor=blue,citecolor=blue]{hyperref}
\usepackage{color}
\usepackage{setspace}
\usepackage{tabularx}
\usepackage{comment}
\usepackage{amsfonts}
\usepackage{paralist}
\usepackage{aliascnt}
\usepackage[initials, lite]{amsrefs}
\usepackage{amscd}
\usepackage{blkarray}
\usepackage{booktabs}
\usepackage{enumitem}
\usepackage{mathbbol}
\usepackage{setspace}
\usepackage[inner=2.5cm,outer=2.5cm, bottom=3.2cm]{geometry}
\usepackage{tikz, tikz-cd}
\usepackage{calligra,mathrsfs}
\usepackage{tikz}
\usetikzlibrary{matrix}
\usetikzlibrary{arrows,calc}
\allowdisplaybreaks

\newtheorem{innercustomthm}{Theorem}
\newenvironment{customthm}[1]
  {\renewcommand\theinnercustomthm{#1}\innercustomthm}
  {\endinnercustomthm}

\newtheorem{theorem}{Theorem}[section]

\newaliascnt{headcor}{headthm}

\aliascntresetthe{headcor}

\newaliascnt{headconj}{headthm}

\aliascntresetthe{headconj}

\newaliascnt{corollary}{theorem}
\newtheorem{corollary}[corollary]{Corollary}
\aliascntresetthe{corollary}

\newaliascnt{claim}{theorem}

\aliascntresetthe{claim}

\newaliascnt{lemma}{theorem}
\newtheorem{lemma}[lemma]{Lemma}
\aliascntresetthe{lemma}

\newaliascnt{conjecture}{theorem}

\aliascntresetthe{conjecture}

\newaliascnt{proposition}{theorem}
\newtheorem{proposition}[proposition]{Proposition}
\aliascntresetthe{proposition}

\theoremstyle{definition}
\newaliascnt{definition}{theorem}
\newtheorem{definition}[definition]{Definition}
\aliascntresetthe{definition}

\newaliascnt{notation}{theorem}
\newtheorem{notation}[notation]{Notation}
\aliascntresetthe{notation}

\newaliascnt{example}{theorem}

\aliascntresetthe{example}

\newaliascnt{examples}{theorem}

\aliascntresetthe{examples}

\newaliascnt{remark}{theorem}
\newtheorem{remark}[remark]{Remark}
\aliascntresetthe{remark}

\newaliascnt{question}{theorem}

\aliascntresetthe{question}

\newaliascnt{questions}{theorem}

\aliascntresetthe{questions}

\newaliascnt{problem}{theorem}

\aliascntresetthe{problem}

\newaliascnt{construction}{theorem}

\aliascntresetthe{construction}

\newaliascnt{setup}{theorem}

\aliascntresetthe{setup}

\newaliascnt{algorithm}{theorem}
\newtheorem{algorithm}[algorithm]{Algorithm}
\aliascntresetthe{algorithm}

\newaliascnt{observation}{theorem}

\aliascntresetthe{observation}

\newaliascnt{defprop}{theorem}

\aliascntresetthe{defprop}

\newtheorem*{acknowledgement}{Acknowledgement}

\newcommand{\Hilb}{\mathrm{Hilb}}
\newcommand{\Hom}{\mathrm{Hom}}
\newcommand{\Spec}{\mathrm{Spec}}
\newcommand{\Proj}{\mathrm{Proj}}

\newcommand{\OO}{\mathscr{O}}
\newcommand{\PP}{\mathbf{P}}

\newcommand{\mm}{\mathfrak{m}}

\newcommand{\GL}{\mathrm{GL}}

\newcommand{\e}{\boldsymbol{e}}
\newcommand{\f}{\boldsymbol{f}}
\newcommand{\kk}{\mathbf{k}}
\newcommand{\A}{\boldsymbol{\alpha}}
\newcommand{\x}{\boldsymbol{x}}

\DeclareFontFamily{OT1}{pzc}{}
\DeclareFontShape{OT1}{pzc}{m}{it}{<-> s * [1.100] pzcmi7t}{}
\DeclareMathAlphabet{\mathchanc}{OT1}{pzc}{m}{it}

\def\equationautorefname~#1\null{(#1)\null}
\def\sectionautorefname~#1\null{Section #1\null}
\def\subsectionautorefname~#1\null{\S #1\null}

\begin{document}
\author[Ritvik \,Ramkumar]{Ritvik~Ramkumar}
\address{(Ritvik Ramkumar) Department of Mathematics\\Cornell University\\ Ithaca, NY\\14850\\USA}
\email{ritvikr@cornell.edu}
\title{Hilbert schemes with two Borel-fixed points}
\maketitle
\begin{abstract} 
We characterize Hilbert polynomials that give rise to Hilbert schemes with two Borel-fixed points and determine when the associated Hilbert schemes or their irreducible components are smooth. In particular, we show that the Hilbert scheme is reduced and has at most two irreducible components. By describing the singularities in a neighbourhood of the Borel-fixed points, we prove that the irreducible components are Cohen-Macaulay and normal. We end by giving many examples of Hilbert schemes with three Borel-fixed points.
\end{abstract}

\setcounter{section}{-1}
\section{Introduction}

The Hilbert scheme $\Hilb^P(\PP^n)$ which parameterizes closed subschemes of $\mathbf{P}^n$ with a fixed Hilbert polynomial $P$ introduced by Grothendieck \cite{g} has attracted a lot of interest, but their global geometry is not well understood. The earliest results in this direction were obtained by Hartshorne \cite{HARTSHORNE_CONNECTED}, who showed that $\Hilb^P(\PP^n)$ is connected, and Fogarty \cite{FOGARTY}, who proved that $\Hilb^{P}(\PP^2)$ is smooth. Later on, Reeves-Stillman \cite{rs} showed that every Hilbert scheme of projective space contains a smooth Borel-fixed point. As a consequence, Hilbert schemes with a single Borel-fixed point are smooth and irreducible, and Staal \cite{STAAL} completely classified these Hilbert schemes. In fact, most other Hilbert schemes or components of Hilbert schemes that are very well understood have few Borel-fixed points. For example, the twisted cubic compactification $\Hilb^{3t+1}(\PP^n)$, which has two smooth components that meet transversely \cite{ps}, has three Borel-fixed points. Thus, by restricting the structure of the Borel-fixed points one might obtain many smooth or mildly singular (components of) Hilbert schemes. 
After the first version of this paper was available, Skjelnes-Smith \cite{SMOOTH_HILB} classified all smooth Hilbert schemes and described their geometry. 
Complementing \cite{SMOOTH_HILB}, our work may be seen as a first step towards a classification of mildly singular Hilbert schemes.

In this paper, we study Hilbert schemes with at most three Borel-fixed points. We classify Hilbert schemes with two Borel-fixed points, show that they are reduced and determine when they are irreducible or smooth. Using the tangent-obstruction theory for the Hilbert scheme  we show that the singularities that occur are cones over certain Segre embeddings of $\PP^a \times \PP^b$. 

To state our results we use the Gotzmann decomposition of a Hilbert polynomial \cite{GOTZMANN}. An \textit{integer partition} $\lambda$ is an $m$-tuple of positive integer $\lambda = (\lambda_1, \lambda_2, \dots, \lambda_m)$ satisfying  $\lambda_1 \geq \lambda_2 \geq \cdots \lambda_m \geq 1$. If $P$ is the Hilbert polynomial of some subscheme of $\PP^n$ then there exists an integer partition $\lambda$ such that
\begin{equation} \label{EXPANSION}
P = P_\lambda:= \sum_{i=1}^m \binom{t+\lambda_i - i}{\lambda_i-1}.
\end{equation}

\begin{customthm}{A}  \label{intoThm}
Let $\text{char}(\kk) = 0$. A Hilbert scheme $\Hilb^{P_\lambda}(\PP^n)$ with two Borel-fixed points is either smooth, irreducible and singular, or a union of two components.
\begin{itemize}[topsep=0pt,itemsep=1ex]
\item It is smooth when 
\begin{enumerate}[label=(\roman*)]
\item $\lambda = (n^s,1,1,1)$ with $n \geq 2$;
\item $\lambda = (2^s,1,1,1,1)$;
\item $\lambda = (n^2,2^q,1)$ with $n > 2$ and $q \geq 4$;
\item  $\lambda = (n^s,(d+1)^q,1)$ with $n > d+1 > 2$ and $q \geq 2$;
\end{enumerate} 
\item It is irreducible and singular when
\begin{enumerate}[label=(\roman*)] \setcounter{enumi}{4}
\item $\lambda=(n^s,(d+1)^q,r+1,1)$ with $n > d+1  > r +1 > 2$;
\item $\lambda=(n^s,(d+1)^q,2,1)$ with $n > d+1  > 2$ and $q \geq 3$.
\item $\lambda = (n^s,d+1,1,1)$ with $n > d+1 > 1$.
\end{enumerate}
\item It is a union of two components when
\begin{enumerate}[label=(\roman*)] \setcounter{enumi}{7}
\item $\lambda = (n^s,2,2,1)$ for $n \geq 3$;
\item $\lambda = (n^s,d+1,2,1)$ with $n > d+1 > 3$.
\end{enumerate}
\end{itemize} 
In all these cases, the irreducible components are normal and Cohen-Macaulay.
\end{customthm}

A more precise statement, including a description of the general member and the singularity type in each case, is given in \autoref{PROOF}. As an illustrative example, consider the partition when $\lambda = (2,1,1)$ in $\PP^3$, which is a special case of \textit{(vii)} with $s=0$. In this case the Hilbert scheme $\Hilb^{P_\lambda}(\PP^3)$  is Gorenstein with general member parameterizing a line union two isolated points.

After the first version of this paper was posted, work of Staal \cite{STAAL_2} shows that the classification of Hilbert schemes with two Borel-fixed points extends to positive characteristics with a minor modification. In particular, \cite[Theorem 1.1]{STAAL_2} states that for $\text{char}(\kk) \ne 2$ the Hilbert scheme $\Hilb^{P_\lambda}(\PP^n)$ has two Borel-fixed points if and only if $\lambda$ is as in one of the cases in \autoref{THEOREM_A}. If $\text{char}(\kk) = 2$ then $\lambda$ can be any of the cases of \autoref{THEOREM_A} except for case \textit{(ii)}. Since our deformation computations are characteristic independent (see \autoref{DeformationSection} and \autoref{PROOF}), we obtain a description of the singularities in all characteristics.


In the last section we collect various examples of Hilbert schemes with three Borel-fixed points that have appeared in the literature. In contrast with \autoref{THEOREM_A}, we show that these Hilbert schemes can have three irreducible components and that the components can meet each other in different ways.

There are many directions one can explore using the techniques developed in this paper. For instance, our methods in \autoref{classifybobo} can be used to feasibly classify Hilbert schemes with a small number of Borel-fixed points. In these cases, experiments in \cite{M2} suggest that the Borel-fixed points have mild singularities. Another idea would be to study deformations of a restricted class of Borel-fixed ideals. For example, one can consider classes of Borel-fixed ideals whose generators have similar properties to the generators of the ideals studied in this paper. In fact, the authors in \cite{LELLA_ROGGERO, CMR} develop a theory of marked schemes and use it to provide local equations around Borel-fixed ideals. In \cite{BCR} this is effectively applied to show that $\Hilb^{16}(\PP^7)$ has at least three irreducible components, and that the graded Gorenstein $\kk$-algebras of Hilbert function $(1,7,7,1)$ are smoothable.



\section{Preliminaries}
In this section we fix our notation and review some well known facts about Borel-fixed points.

\subsection{Notation}  Let $\kk$ be an algebraically closed field. We use $S$ to denote the polynomial ring $\kk[x_0,\dots,x_n]$ and $\mathfrak{m}:=(x_0\dots,x_n)$ to denote its maximal ideal. We denote the monomial $x_0^{\alpha_0}\cdots x_{n}^{\alpha_n}$ by $\boldsymbol{x}^{\boldsymbol{\alpha}}$. We use $S_d$ to denote the subspace generated by monomials of degree $d$. The \textit{support} of a monomial is the set of all variables that divide the monomial. By lexicographic ordering we will mean the standard lexicographic ordering on $S$ with $x_0 > x_1 > \cdots > x_n$.

All ideals are assumed to be saturated unless otherwise specified. We will use $[I]$ or $[X]$ to denote the $\kk$-point in the Hilbert scheme corresponding to $X = \Proj(S/I)$. We use $P_X(t)$ or $P_{S/I}(t)$ to denote the Hilbert polynomial of the subscheme $X= \Proj \, (S/I) \subseteq \mathbf{P}^n$. We sometimes call this the Hilbert polynomial of $I$.

We use $\lambda$ to denote the tuple $(\lambda_1, \lambda_2, \dots, \lambda_m)$ of weakly decreasing positive integers and call it an \textit{integer partition}. We use $P_\lambda$ to denote the Hilbert polynomial (Equation \autoref{EXPANSION}) associated to $\lambda$. 
Hilbert scheme are indexed by partitions $\lambda$ and we will do this implicitly by writing them as $\Hilb^{P_\lambda}(\PP^n)$. The dimension of the subscheme with Hilbert polynomial $P_\lambda$ is $\lambda_1-1$. In particular, if the closed subscheme is proper and non-empty we have $1 \leq \lambda_1 \leq n$. For more details we refer to \cite[\S 3]{SMOOTH_HILB}.

\subsection{Borel-fixed points} Given a matrix $A=(a_{ij})_{ij}\in \GL(n+1)$, the map on variables $x_{i}\mapsto\sum a_{ij}x_{j}$ induces an action on the set of ideals of $S$ with Hilbert polynomial $P(t)$. Thus, the group $\GL(n+1)$ acts on $\Hilb^P(\PP^n)$ and so does its subgroup, $\mathcal{B}$, of upper triangular matrices. A closed point (resp. ideal) is said to be \textit{Borel-fixed} if it is fixed by the subgroup $\mathcal{B}$.

\begin{lemma} \label{borel2} \label{borel} The Hilbert scheme $\Hilb^{P}(\PP^n)$ is reduced or smooth if and only if it is reduced or smooth at all the Borel-fixed points, respectively. Moreover, an integral component, $H$, of the Hilbert scheme is normal, Cohen-Macaulay, Gorenstein or smooth if and only if it is normal, Cohen-Macaulay, Gorenstein or smooth at all the Borel-fixed points on $H$, respectively.
\end{lemma}
\begin{proof} 
Given a $\kk$-point $[Z] \in \Hilb^{P}(\PP^n)$, write $\mathcal{B}(Z)$ for the orbit of $Z$ under $\mathcal{B}$. By the Borel fixed-point theorem the closure, $\overline{\mathcal{B}(Z)}$, contains a Borel-fixed point. Assume that the Hilbert scheme is reduced at all the Borel-fixed points. Since the reduced locus is open, a non-empty open subset of $\overline{\mathcal{B}(Z)}$ is also reduced. Thus, some element of $\mathcal{B}(Z)$ is also non-reduced. Since $\mathcal{B}$ acts by automorphisms, $Z$ must be a reduced point. The same proof works for smoothness as the smooth locus is also open.

The action of $\mathcal{B}$ restricts to any irreducible component of the Hilbert scheme. Since the normal, Cohen-Macaulay and Gorenstein loci are all open, the proof given in the previous paragraph also proves the second statement.
\end{proof}

Since Borel-fixed ideals are fixed by the set of diagonal matrices, they must be monomial ideals. A monomial ideal $I \subseteq S$ is said to be \textit{strongly stable} if for any monomial $\boldsymbol{x}^{\boldsymbol{\alpha}} \in I$ divisible by $x_j$ we have $\boldsymbol{x}^{\boldsymbol{\alpha}} \frac{x_i}{x_j} \in I$ for all $i < j$. The relation between these two concepts is given by the following theorem.

\begin{proposition}[\textup{\cite[Proposition 2.7]{BAYER_STILLMAN} \label{exchange}}] If $\text{char}(\kk)= 0$ a monomial ideal $I \subseteq S$ is Borel-fixed if and only if $I$ is strongly stable.
\end{proposition}

\subsection{Resolutions of strongly stable ideals} \label{Eliahou-Kervaire}
The Eliahou-Kervaire resolution provides an explicit minimal free resolution of a strongly stable ideal \cite[Section 2]{EK}. We will mostly be interested in resolutions of ideals of the form $I = x_0(x_0,\dots,x_{n-1}) + x_1^q(x_1,\dots,x_{p})$ with $q\geq 1$ and $n-1\geq p \geq 0$. Note that $I$ is strongly stable in all characteristics. Following the presentation in  \cite[Section 2]{sp}, let $ 0 \to F_{n-1} \xrightarrow{\psi_{n-1}} \cdots \xrightarrow{\psi_2} F_1 \xrightarrow{\psi_1} F_0 \xrightarrow{\psi_0} I \to 0$ denote the Eliahou-Kervaire resolution of $I$ where
 \begin{equation*}
F_{0} = \left(\bigoplus_{i=0}^{n-1} S(-2)\e_{0i}^{\star}\right) \bigoplus \left(\bigoplus_{i=1}^{p} S(-q-1)\e_{1i}^{\star}\right)
\end{equation*}
and 
\begin{equation*}
F_{1} = \left( \bigoplus_{0 \leq j < i \leq n-1} S(-3)\e_{0i}^{j}  \right) \bigoplus \left( \bigoplus_{0 \leq j < i \leq p} S(-q-2)\e_{1i}^{j}  \right). 
\end{equation*}
The first two differentials are given by $\psi_0(\e_{0i}^{\star}) = x_0x_i$, $\psi_0(\e_{1i}^{\star}) = x_1^{q}x_i$ and, 
\begin{eqnarray*}
\psi_1(\e_{0i}^{j}) & = & x_j\e_{0i}^{\star} - x_i\e_{0j}^{\star}, \quad 0 \leq j < i \leq n-1 \\ 
\psi_1(\e_{1i}^{0}) & = & x_0\e_{1i}^{\star} - x_1^q\e_{0i}^{\star}, \quad 1 \leq i \leq p\\
\psi_1(\e_{1i}^{j}) & = & x_j\e_{1i}^{\star} - x_i\e_{1j}^{\star}, \quad 1 \leq j < i \leq p.
\end{eqnarray*}
This presentation also allows us to explicitly describe the first two terms of the cotangent complex \cite[Chapter 3]{deformationtheory}.
Let $R = S/I$ and let
\begin{equation*}
\text{Kos} := \psi_1^{-1}\left(\{\psi_0(\e_{l_1j_1}^{\star})\e_{l_1j_1}^{\star} - \psi_0(\e_{l_2j_2}^{\star})\e_{l_2j_2}^{\star} \} \right) \subseteq F_1,
\end{equation*}
be the pre-image of the Koszul relations in $F_0$. Let $\psi_1^{\vee}:\Hom_S(F_0,S) \to \Hom_S(F_1,S)$ denote the dual of $\psi_1$. The second cotangent cohomology, $T^2(R/\kk,R)$, is the cokernel of the following map
\begin{equation*}
\Hom_R(F_0\otimes R,R) \xrightarrow{\overline{\psi_1^{\vee}}} \Hom_R\left(F_1/(\mathrm{ker} \, \psi_1 + \text{Kos}),R\right).  \\
\end{equation*}

\subsection{Lexicographic ideals} \label{LEX}
Every Hilbert scheme has a distinguished Borel-fixed ideal called the \textit{lexicographic} ideal. We review its properties following the notation of \cite[\S 3]{SMOOTH_HILB}. A monomial ideal $L \subseteq S$ is a {lexicographic} ideal if, for all integers $j$, the homogeneous component of $I_j$ is the $\kk$-vector space spanned by the $\dim_\kk I_j$ largest monomials in lexicographic order. For an integer partition $\lambda$, there is a unique saturated lexicographic ideal, denoted by $L(\lambda)$, with Hilbert polynomial $P_{\lambda}$ \cite{MACAULAY}. It is also a smooth point in the Hilbert scheme {\cite[Theorem 1.4]{rs} and the irreducible component containing it is called the \textit{lexicographic component}.  To describe this point, let $a_j$ be the number of parts in $\lambda$ equal to $j$ for all $j \in \mathbb{N}$. If $n \geq \lambda_1$ we have
\begin{equation} \label{LEX_POINT}
L(\lambda) := (x_0^{a_n+1},x_0^{a_n}x_1^{a_{n-1}+1}, \dots,
	x_0^{a_n}x_1^{a_{n-1}}\cdots x_{n-3}^{a_3}x_{n-2}^{a_2+1},
	x_0^{a_n}x_1^{a_{n-1}} \cdots x_{n-2}^{a_2}x_{n-1}^{a_1}).
\end{equation}
If the Hilbert scheme has exactly two Borel-fixed points we will use $I(\lambda)$ to denote the non lexicographic Borel-fixed point.

\section{Classifying Hilbert polynomials} \label{classifybobo} In this section we classify Hilbert polynomials with two Borel-fixed ideals in characteristic $0$ (\autoref{poly1} and \autoref{poly2}). The first step is to reduce to studying Hilbert schemes corresponding to integer partitions $\lambda$ with $n > \lambda_1 $, equivalently Hilbert schemes parameterizing subschemes of codimension at least $2$. Using the classification of Hilbert schemes with a single Borel-fixed ideal and \autoref{moo} we obtain the desired classification.

\begin{lemma}\label{DETACH} Let $\lambda = (n^s,\lambda_{s+1},\lambda_{s},\dots,\lambda_m)$ be an integer partition with $s >0$. Then there is an isomorphism
$$
\Hilb^{P_\lambda}(\PP^n) \simeq \PP(H^0(\OO_{\PP^n}(s))) \times \Hilb^{P_{\lambda'}}(\PP^n)
$$
where $\lambda' = (\lambda_{s+1},\dots,\lambda_m)$. This isomorphism is $\GL(n+1)$-equivariant and thus induces a bijection on Borel-fixed ideals, given by $I \mapsto x_0^sI'$.
\end{lemma}
\begin{proof} By  \cite[Theorem 1.4]{FOGARTY} and  \cite[Remark 2, p. 514]{FOGARTY}  there is an isomorphism 
\begin{equation} \label{MULT}
\PP(H^0(\OO_{\PP^n}(s'))) \times \Hilb^{P'}(\PP^n) \simeq \Hilb^{P_\lambda}(\PP^n), \quad (f,[I]) \mapsto  [fI] 
\end{equation}
where $\deg P' < n-1$ and 
$$
P_\lambda(t) = \binom{t+n}{n} - \binom{t+n-s'}{n} + P'(t-s').
$$
Since the morphism \autoref{MULT} is given by multiplication of ideals, it is also $\GL(n+1)$-equivariant. Using the well-known identity on summation of binomial coefficients we obtain
$$
\sum_{i=1}^s \binom{t+n-i}{n-1} +\sum_{i=s+1}^{m} \binom{t+\lambda_i-i}{\lambda_i-1} 
	= P_\lambda(t) 
	= \sum_{i=1}^{s'} \binom{t+n-i}{n-1} + P'(t-s').
$$
Since $\deg P' < n-1$ we must have $s= s'$ and this, in turn, implies that $P' = P_{\lambda'}$. The desired bijection on Borel-fixed points follows from the $GL(n+1)$-equivariance.
\end{proof}

By \autoref{DETACH} it suffices to classify Borel-fixed ideals in Hilbert schemes corresponding to $\lambda$ with $n>\lambda_1$.

\begin{notation} For the rest of this section we will assume $\text{char}(\kk) = 0$.
\end{notation}

We begin by briefly describing a procedure that generates all the Borel-fixed ideals in characteristic $0$. Following \cite{CLMR, LELLA_12},
we fix an order on the variables so that $x_0 > x_1 > \cdots > x_n$. This induces a partial order on monomials of a fixed degree: if $
x_i > x_j$ then
$ x_i\x^{\A} > x_j \x^{\A}$.
This is called the Borel order and we denote it by $\geq_B$. 

Let $I \subseteq S$ be a stongly stable ideal with Hilbert polynomial $P(t)$ and let $\mathcal{G}(I)$ denote the set of minimal generators of $I$. 
Given an element $\x^{\A}$ of $\mathcal{G}(I)$ that is also minimal with respect to $\geq_B$ one can produce a new strongly stable ideal with Hilbert polynomial $P(t)+1$. This procedure is known as an \textit{expansion} of $I$ with respect to $\x^{\A}$, and the new strongly stable ideal is generated by $$(\mathcal{G}(I) \setminus \{\x^{\A} \}) \cup \{\x^{\A}x_r, \x^{\A}x_{r+1},\dots,\x^{\A}x_{n-1} \}$$
where $r=\max \{i:x_i | \x^{\A}\}$. For our purposes, we just need the penultimate step in the recursive algorithm.

\begin{algorithm}  \label{moo}  Every saturated strongly stable ideal of $S$ with Hilbert polynomial $P(t)$ is obtained from a strongly stable ideal of $R = \kk[x_0,\dots,x_{n-1}]$ with Hilbert polynomial $\Delta P(t) := P(t) - P(t-1)$ via a sequence of expansions. More precisely, $I$ is obtained by successively expanding $JS$ $c$ times, where $J$ is a strongly stable ideal of $R$ with Hilbert polynomial $\Delta P(t)$ and $c= P(t) - P_{S/JS}(t)$ is a constant.
\end{algorithm}

\begin{remark} An alternative algorithm to generate the strongly stable ideals is presented in \cite{moo}.
\end{remark}

Implicit in the above Algorithm is the following Lemma that will be extremely useful for us.

\begin{lemma}[\textup{\cite[Lemma 3.1, \S 4.2]{LELLA_12}}] \label{move} Let $I \subseteq S$ be a saturated strongly stable ideal. Then we can always expand $I$ at a minimal generator of degree $e$ that is minimal w.r.t. to $\geq_B$. Any such expansion is strongly stable with Hilbert polynomial $P_{S/I}(t) +1$.
\end{lemma}

\begin{remark} Integer partitions behave well with respect to the difference operator. If $\lambda = (\lambda_1,\dots,\lambda_m,1^s)$ then we have $\Delta^1 P_\lambda = P_{\lambda''}$ where $\lambda'' = (\lambda_1-1,\dots,\lambda_m-1)$. Indeed, we have
$$
\Delta^1 P_\lambda = 
	\sum_{i=1}^{m+s} \binom{t+\lambda_i-i}{\lambda_i-1} - \sum_{i=1}^{m+s} \binom{t-1+\lambda_i-i}{\lambda_i-1} =
	\sum_{i=1}^{m+s} \binom{t+(\lambda_i-1)-i}{(\lambda_i-1)-1} = P_{\lambda''}.
$$
\end{remark}

By our discussion above we see that the number of Borel-fixed points on a Hilbert scheme $\Hilb^{P_{\lambda}}(\PP^n)$ are, to some extent, determined by the number of Borel-fixed points on $\Hilb^{\Delta P_{\lambda}}(\PP^{n-1})$ and $\Hilb^{P_{\lambda}-1}(\PP^n)$. It turns out that by considering $\Hilb^{P_{\lambda}-1}(\PP^n)$, we can greatly restrict the partitions $\lambda$ that could give rise to Hilbert schemes with two Borel-fixed points.

\begin{lemma} \label{minus1} If $\Hilb^{P_\lambda}(\PP^n)$ has more than one Borel-fixed point, then $\Hilb^{P_\lambda-1}(\PP^n)$ is non-empty.
\end{lemma}
\begin{proof} Let $\lambda = (\lambda_1,\lambda_2,\dots,\lambda_m)$. If $\Hilb^{P_\lambda}(\PP^n)$ has more than one more Borel-fixed point then \cite[Theorem 1.1]{STAAL} implies that $\lambda_m = 1$ and $m \geq 2$. It follows that  
\begin{equation*}
P_\lambda-1 = \sum_{i=1}^{m}\binom{t+\lambda_i-i}{\lambda_i-1} - 1 
			= \sum_{i=1}^{m-1}\binom{t+\lambda_i-i}{\lambda_i-1}
			= P_{\lambda'}
\end{equation*}
with $\lambda' = (\lambda_1,\dots,\lambda_{m-1})$. Since $\lambda'$ is an integer partition with $1 \leq \lambda'_1 \leq n$, the result follows.
\end{proof}

We can now state a necessary condition for a Hilbert scheme to have two Borel-fixed points.

\begin{proposition} \label{4poly} Let $\lambda = (\lambda_1,\lambda_2,\dots,\lambda_m)$ be an integer partition with $\lambda_1 \leq n-1$. If $\Hilb^{P_\lambda}(\PP^n)$ has two Borel-fixed points then $\lambda = ((d+1)^q,1) $ or $\lambda = ((d+1)^q,r+1,1)$
\end{proposition}
\begin{proof} By \cite[Theorem 1.1]{STAAL} we may assume $\lambda_m = 1$ and $m \geq 2$. Let $\lambda' = (\lambda_1,\dots,\lambda_{m-1})$ and we have $P_\lambda = P_{\lambda'}+1$. If the lexicographic point, $L(\lambda')$, was generated in more than two degrees then \autoref{move} would imply that $\Hilb^{P_\lambda}(\PP^n)$ contains at least three Borel-fixed points; a contradiction. So we may assume that $L(\lambda')$ (Equation \autoref{LEX_POINT}) is generated in at most two degrees. Let $r$ be the smallest integer for which $a_{r+1} \ne 0$ and $d$ be the largest integer for which $a_{d+1} \ne 0$. By assumption we have $a_n = 0$. If $r = d$ we must have 
\begin{equation} \label{FIRST_LEX}
L(\lambda') = (x_0,\dots,x_{n-d-2},x_{n-d-1}^{a_{d+1}})
\end{equation}
which implies $\lambda' = ((d+1)^{a_{d+1}})$. If $d > r$ we have $a_{d+1}+1 = a_{d+1} + a_{d} + 1 = \cdots = a_{d+1} + \cdots + a_{r+2} + 1 = a_{d+1} + \cdots + a_{r+1}$. This implies $a_{r+2},\dots,a_{d} = 0$ and $a_{r+1} = 1$, and we obtain
\begin{equation} \label{SECOND_LEX}
L(\lambda') = (x_0,\dots,x_{n-d-2}) +x_{n-d-1}^{a_{d+1}}(x_{n-d-1},x_{n-d-2},\dots,x_{n-r-1})
\end{equation}
and $\lambda' = ((d+1)^{a_{d+1}},r+1)$, as required.
\end{proof}

We now turn our attention to eliminating some of the $\lambda$ that appeared in \autoref{4poly}. If the Hilbert scheme $\Hilb^{P_\lambda}(\PP^n)$ has two Borel-fixed points then they are both on the lexicographic component. Let $X_1$ and $X_2$ denote the two Borel-fixed points. By \cite[Theorem 11]{REEVES} the hyperplane sections $X_i \cap V(x_n)$ must be equal to the lexicographic point $V(L(\lambda'))$ where $\Delta P_{\lambda} = P_{\lambda'}$. Thus, if we produce a Borel-fixed point on $\Hilb^{P_\lambda}(\PP^n)$ whose hyperplane section is not $L(\lambda')$, then the corresponding Hilbert scheme cannot have two Borel-fixed points. Of course, sometimes it is simpler to directly construct three Borel-fixed ideals. We use both of these methods to obtain the following Lemma.

\begin{lemma} \label{basecase} Let $\lambda = (\lambda_1,\lambda_2,\dots,\lambda_m)$ be an integer partition with $\lambda_1 \leq n-1$. For the following partitions $\lambda$, the Hilbert scheme $\Hilb^{P_\lambda}(\PP^n)$ has at least three Borel-fixed points
\begin{enumerate}
\item[i)]  $\lambda = (1^b)$ with $b \geq 4$ and $n \geq 3$,
\item[ii)] $\lambda = (1^b)$ with $b \geq 5$  and $n = 2$,
\item[iii)] $\lambda = ((d+1)^2,2,1)$ with $d\geq 1$,
\item[iv)] $\lambda = ((d+1)^q,1^2)$ with  $d \geq 1$ and $q > 1$.
\end{enumerate}
\end{lemma}

\begin{proof} For the rest of the proof let $R = \kk[x_0,\dots,x_{n-1}]$. In case \textit{i)} we may use \autoref{exchange} to verify that the following ideals are Borel-fixed with Hilbert polynomial $P_\lambda = b$: \begin{equation*}
\begin{aligned}
&(x_0,\dots,x_{n-3},x_{n-2},x_{n-1}^b), \, (x_0,\dots,x_{n-3},x_{n-2}^2,x_{n-2}x_{n-1},x_{n-1}^{b-1}) \text{ and }\\
&(x_0,\dots,x_{n-4},x_{n-3}^2,x_{n-3}x_{n-2},x_{n-3}x_{n-1},x_{n-2}^2,x_{n-2}x_{n-1},x_{n-1}^{b-2}).
\end{aligned}
\end{equation*}

Similarly, in case \textit{ii)} we may use \autoref{exchange} to verify the following ideals are Borel-fixed with Hilbert polynomial $P_\lambda = b$:
$$
(x_0,x_1^b), \, (x_0^2,x_0x_1,x_1^{b-1})  \text{ and }  (x_0^2,x_0x_1^2,x_1^{b-2}).
$$

If we are in case \textit{iii)} then consider the following Borel-fixed ideal
$$
J = (x_0,\dots,x_{n-d-3})+x_{n-d-2}(x_{n-d-2},\dots, x_{n-2}) +(x_{n-d-1}^2).
$$
To see that $J$ has Hilbert polynomial $P_\lambda$, it suffices to compare it to 
$$
L(\lambda) = (x_0,\dots,x_{n-d-2}) +x_{n-d-1}^{2}(x_{n-d-1},x_{n-d-2},\dots,x_{n-3}) + x_{n-d-1}^{2}x_{n-2}(x_{n-2},x_{n-1}).
$$
Indeed, for $j \gg 0$ we have 
$$
J_j \setminus L(\lambda)_j = \{x_{n-d-1}^2x_{n-d-2}x_n^{j-3}\} \cup \{x_{n-d-1}^2x_{n-1}^ex_n^{j-2-e}\}_{0\leq e \leq j-2}
$$ 
and 
$$
L(\lambda)_j \setminus J_j = \{x_{n-d-2}x_{n-1}^ex_{n}^{j-1-e}\}_{0\leq e \leq j-1}.
$$
Since these two sets have the same cardinality $j$, it follows that $P_{S/L(\lambda)}(t) = P_{S/J}(t)$. The hyperplane section $V(x_n) \cap V(J)$ is defined by the saturated ideal
$$ (x_0,\dots,x_{n-d-3})+x_{n-d-2}(x_{n-d-2},\dots, x_{n-2}) +(x_{n-d-1}^2).
$$
Since this is different from $L(d^3) = (x_0,\dots,x_{n-d-2},x_{n-d-1}^3)$, the Hilbert scheme cannot have two Borel-fixed points.

Finally, if we are in case \textit{iv)} we have the following Borel-fixed ideals
\begin{align*}
L(\lambda)  &= (x_0,\dots,x_{n-d-2}) +x_{n-d-1}^{q}(x_{n-d-1},x_{n-d-2},\dots,x_{n-2}) + (x_{n-d-1}^{q}x_{n-1}^2), \\
I & = (x_0,\dots,x_{n-d-3}) + x_{n-d-2}(x_{n-d-2},\dots,x_{n-1}) +x_{n-d-1}^{q}(x_{n-d-1},x_{n-d-2},\dots,x_{n-1}), \\
J & = (x_0,\dots,x_{n-d-3}) + x_{n-d-2}(x_{n-d-2},\dots,x_{n-2},x_{n-1}^2) + (x_{n-d-1}^q).
\end{align*}
Just as we did in case \textit{iii)}, it is straightforward to see that the three ideals have Hilbert polynomial $P_\lambda$. For instance, consider $J$ and note that for $j \gg 0$ we have
$$
J_j \setminus L(\lambda)_j = \{x_{n-d-1}^qx_{n-1}x_n^{j-q-1}\} \cup \{x_{n-d-1}x_n^{j-q}\}
$$ 
and 
$$
L(\lambda)_j \setminus J_j = \{x_{n-d-2}x_{n-1}x_n^{j-2}, x_{n-d-2}x_n^{j-1}\}. 
$$
\end{proof}

We are now ready to prove the main result of this section. It will turn out that the constraints we have found on $\lambda$ up until this point are sufficient. We establish this by studying the expansions of Borel-fixed ideals with Hilbert polynomial $\Delta P(t)$ (\autoref{moo}). Since the Borel-fixed ideals naturally fit into two distinct families, we split the result into two Propositions.

\begin{proposition} \label{poly1} Let $\lambda = ((d+1)^q,1)$ with $n -2 \geq d$. The Hilbert scheme $\Hilb^{P_\lambda}(\PP^n)$ has two Borel-fixed points if and only if $n \geq 2$ and 
\begin{enumerate}
\item[i)] $d=0$ and $q=2$, or
\item[ii)] $d = 0$, $q=3$ and $n=2$, or
\item[iii)] $d=1$ and $q \ne 1,3 $, or
\item[iv)] $d \geq 2$ and $q \geq 2$.
\end{enumerate}
The two Borel-fixed ideals are
\begin{align*}
I(\lambda) & = (x_0,\dots,x_{n-d-3}) + x_{n-d-2}(x_{n-d-2},\dots,x_{n-1}) +(x_{n-d-1}^{q}), \\
L(\lambda) & =  (x_0,\dots,x_{n-d-2}) +x_{n-d-1}^{q}(x_{n-d-1},x_{n-d-2},\dots,x_{n-1}).
\end{align*}
\end{proposition}

\begin{proof} The ideals $I(\lambda)$ and $L(\lambda)$ are expansions of a lexicographic ideal $(x_0,\dots,x_{n-d-2},x_{n-d-1}^q)$. Since the latter ideal has Hilbert polynomial $P_{((d+1)^q)}$, it follows from \autoref{move} that the Hilbert polynomial of $I(\lambda)$ and $L(\lambda)$ is $P_\lambda$. We first show that the cases are necessary. By \cite[Theorem 1.1 (ii)]{STAAL} if $n=1$ or $q = 1$ the Hilbert scheme has a single Borel-fixed point. The remaining conditions on $\lambda$ follow from \autoref{basecase}.

If we are in case \textit{i)} then the Hilbert scheme parameterizes subschemes of length three. Any such subscheme can be realized as $\lim_{t \to 0} Z_t = Z$ where $Z_t$ a reduced union of three points for $t \in \mathbf{A}^1 - 0$ \cite{cn}. By upper-semicontinuity, since the union of three reduced points is contained in a $\PP^2$, the subscheme $Z$ is also contained in a $\PP^2$. If $Z$ was Borel-fixed this implies $I_Z = (x_0,\dots,x_{n-3}) + JS$ with $J \subseteq S':=\kk[x_{n-2},x_{n-1},x_{n}]$ and $P_{S'/J}(t)=3$.
Using \autoref{exchange} we see that only choices are $(x_0,\dots,x_{n-3},x_{n-2}^2,x_{n-2}x_{n-1},x_{n-1}^2)$ and $(x_{0},\dots,x_{n-3},x_{n-2},x_{n-1}^3)$.

If we are in case \textit{ii)} then \autoref{exchange} shows that $(x_0,x_1^4)$ and $(x_0^2,x_0x_1,x_1^3)$ are the only two Borel-fixed ideals. 

So we may assume that we are in case \textit{iii)} or case \textit{iv)} of the theorem. Let $\lambda' = ((d+1)^q)$ and $\lambda'' = (d^q)$. By \autoref{moo} we begin by computing all the Borel-fixed ideals in $R:= \kk[x_0,\dots,x_{n-1}]$ with Hilbert polynomial, $\Delta^1 P_\lambda = P_{\lambda''}$.

For $d \geq 2$ the Hilbert scheme $\Hilb^{P_{\lambda''}}(\Proj(R))$ has a unique Borel-fixed point \cite[Theorem 1.1]{STAAL} and it is given by 
$
L(\lambda'') = (x_0,\dots,x_{n-d-2},x_{n-d-1}^q).
$
The lift of $L(\lambda'')$ to $S$ is just the lexicographic ideal, $L(\lambda')$, with Hilbert polynomial $P_{\lambda'} = P_\lambda - 1$. Thus, in the last step of the algorithm, we only need to perform \textit{one} successive expansion. Once with the monomial $x_{n-d-2}$ and once with the monomial $x_{n-d-1}^{q}$, giving us the two desired Borel-fixed ideals.

The last case is if $d=1$ and $q \ne 1,3$. In this case we have 
\begin{equation*}
P_\lambda(t) = 
\sum_{i=1}^{q} \binom{t+2-i}{2-1} + 1 
= qt + 2 - \binom{q-1}{2}.
\end{equation*}
Since $\Delta^1(P_\lambda) = q$ we compute all the Borel-fixed ideals in $R$ with Hilbert polynomial $q$. One such ideal is $I= (x_0,\dots,x_{n-3},x_{n-2}^{q})$ whose lift, $IS$, is the ideal of a plane curve of degree $q$. Thus, the Hilbert polynomial of $IS$ is $P_{\lambda'}$ and we may expand $IS$ at $x_{n-3}$ and $x_{n-2}^q$ to obtain the two Borel-fixed ideals. To finish, it suffices to show that if $J$ is a Borel-fixed ideal in $R$ different from $I$ then the Hilbert polynomial of the lift, $JS$, is bigger than $P_{\lambda}$. For such a $J$ to exist we must have $q \geq 4$. 
In particular, we will prove that  $P_{S/JS}(t) \geq P_\lambda(t) +1 = P_{\lambda'}(t) + 2 $ for all $t \gg 0$. Since $J \ne I$, we may assume that $x_{n-2}^\ell \in J$ and $x_{n-2}^{\ell-1} \notin J$ for some  $1<\ell < q$. 
This implies that for $j \gg 0$, $(R/J)_j$ is spanned by 
\begin{equation*}
\left\{ m_1x_{n-1}^{j-\deg m_1},\dots,m_{q-\ell}x_{n-1}^{j-\deg m_{q-\ell}},x_{n-1}^j,x_{n-2}x_{n-1}^{j-1},\dots,x_{n-2}^{\ell-1}x_{n-1}^{j-\ell +1} \right\}.
\end{equation*}
We may assume that the $m_i$ are monomials of degree strictly less than $\ell$ and not divisible by $x_{n-1}$ (applying the exchange property to $x_{n-2}^{\ell}$, we see that $J$ contains all monomials of degree at least $\ell$ supported on $x_0,\dots,x_{n-2}$). 
Thus, for $j \gg 0$ the graded piece $(S/JS)_j$ contains the monomials in $x_{n-2}^p(x_{n-1},x_n)^{j-p}$ for $0\leq p \leq \ell -1$ and the monomials in $m_v(x_{n-1},x_{n})^{j-\deg m_v}$ for $1 \leq v \leq q -\ell$. This implies 
\begin{equation*}
\dim_\kk (S/J)_j  \geq \sum_{p=0}^{\ell-1} (j-p+1) + \sum_{v = 1}^{q-\ell} (j-\deg m_v+1) 
 \geq  \sum_{p=0}^{\ell-1} (j-p+1) + \sum_{v = 1}^{q-\ell} (j-\ell+1+1). 
\end{equation*}
If we further assume $\ell < q -1$, we may rewrite the sum and obtain
\begin{eqnarray*}
\dim_\kk (S/J)_j & \geq &  \sum_{p=0}^{\ell-1} (j-p+1) + \sum_{v = 1}^{q-\ell} (j-\ell+1) + (q -\ell)  \\
 &\geq &   \sum_{p=0}^{\ell-1} (j-p+1) + \sum_{v = \ell}^{q-1} (j-v+1) + (q -\ell) \\
	& = & \sum_{p=0}^{q-1}(j-p+1) + (q - \ell) \\
	&= & qj + 1 - \binom{q-1}{2} + (q-\ell) \\
	& \geq &  \dim_\kk (S/IS)_j + 2 = P_{\lambda'}(j) + 2
\end{eqnarray*}
as required. Finally, if $\ell = q -1$, the exchange property forces
\begin{equation*}
J= (x_0,\dots,x_{n-4},x_{n-3}^2,x_{n-3}x_{n-2},x_{n-2}^{q-1}).
\end{equation*}
Since $q \geq 4$, one can observe that $P_{S/JS}(t) = P_{S/IS}(t)+2$, completing the proof.
\end{proof}

\begin{proposition} \label{poly2} Let $\lambda = ((d+1)^q,r+1,1)$ with $d > r$. The Hilbert scheme $\Hilb^{P_\lambda}(\PP^n)$ has two Borel-fixed points if and only if $n \geq 2$ and
\begin{enumerate}
\item[i)] $r = 0, q =1$, or
\item[ii)] $r = 1, q \ne 2$, or
\item[iii)] $r \geq 2$ .
\end{enumerate}
The two Borel-fixed ideals are
\begin{align*}
I(\lambda) & = (x_0,\dots,x_{n-d-3}) + x_{n-d-2}(x_{n-d-2},\dots,x_{n-1}) +x_{n-d-1}^{q}(x_{n-d-1},x_{n-d-2},\dots,x_{n-r-1}), \\
L(\lambda) & = (x_0,\dots,x_{n-d-2}) +x_{n-d-1}^{q}(x_{n-d-1},x_{n-d-2},\dots,x_{n-r-2}) + x_{n-d-1}^{q}x_{n-r-1}(x_{n-r-1},\dots,x_{n-1}).
\end{align*}
\end{proposition}
\begin{proof} Since $I(\lambda)$ and $L(\lambda)$ are expansions of the lexicographic ideal (\ref{SECOND_LEX}) it follows from \autoref{move} that their Hilbert polynomial is $P_\lambda$. By \autoref{basecase} these conditions are also necessary; if $n=1$ the Hilbert scheme has a single Borel-fixed point.

Now assume that we are in case \textit{i)}, \textit{ii)} or \textit{iii)}. Let $\lambda' = ((d+1)^q,r+1)$ and $\lambda'' = (d^q,r)$. We begin by  computing all the Borel-fixed ideals in $R:=\kk[x_0,\dots,x_{n-1}]$ with Hilbert polynomial $\Delta^1 P_\lambda = P_{\lambda''}$. 

If $r \geq 2$ or $(r,q) = (1,1)$ the Hilbert scheme $\Hilb^{P_{\lambda''}}(\Proj(R))$ has a unique Borel-fixed point \cite[Theorem 1.1]{STAAL} and it is given by 
$
L(\lambda'') = (x_0,\dots,x_{n-d-2}) +x_{n-d-1}^{q}(x_{n-d-1},x_{n-d-2},\dots,x_{n-r-1}).
$
The lift of $L(\lambda'')$ to $S$ is just the lexicographic ideal, $L(\lambda')$, with Hilbert polynomial $P_{\lambda'} = P_\lambda - 1$. Thus, to obtain all the Borel-fixed ideals we only need to perform a single expansion. Once with the monomial $x_{n-d-2}$ and once with the monomial $x_{n-d-1}^{q}x_{n-r-1}$, giving us the two Borel-fixed ideals.

Similarly, if $(r,q)=(0,1)$ the Hilbert scheme $\Hilb^{P_{\lambda''}}(\Proj(R))$ has a unique Borel-fixed point \cite[Theorem 1.1]{STAAL} and it is given by $(x_0,\dots,x_{n-d-1})$. The lift to $S$  has Hilbert polynomial $\binom{t+d}{d} = P_{\lambda} -2$. Thus, we begin by performing an expansion with $x_{n-d-1}$ to obtain $(x_0,\dots,x_{n-d-2}) + x_{n-d-1}(x_{n-d-1}, \break \dots,x_{n-1})$. This is the lexicographic ideal $L(\lambda')$ and we conclude as in the previous paragraph.

Assume  $r=1$ and $q \geq 3$. Then  \autoref{poly1} \textit{iii)} implies that the Hilbert scheme $\Hilb^{P_{\lambda''}}(\Proj(R))$ has two Borel-fixed ideals, $I'' := (x_0,\dots,x_{n-d-3}) + x_{n-d-2}(x_{n-d-2},\dots,x_{n-2}) +(x_{n-d-1}^{q})$ and $L(\lambda'')$.  We first show that the Hilbert polynomial of $I''S$ is larger than $P_\lambda$. We can do this by comparing the number of generators of $(I''S)_j$ to those of $I(\lambda)_j$ for $j \gg 0$. Let $\mathfrak{C}_j$ denote the intersection of the monomials of $(I''S)_j$ with the monomials of $I(\lambda)_j$. Then it is evident that 
$I(\lambda)_j$
is generated by 
$\mathfrak{C}_j \cup \{x_{n-d-2}x_{n-1}x_{n-1}^ax_{n}^b\}_{a+b = j-2}$
while
$(I''S)_j$
is generated by
$\mathfrak{C}_j \cup \{x_{n-d-1}^qx_{n-1}^ax_n^b\}_{a+b=j-q}$
for all $j \gg 0$. This implies $P_{S/I(\lambda)}(t) + j-1 =  P_{S/I''S}(t)+ j-q+1$. It follows that $P_{S/I''S}(t) = P_{S/I(\lambda)}(t) + (q-2) = P_\lambda(t) + (q-2) > P_\lambda(t)$, as required.
Thus, we only need to perform one successive expansion of the lexicographic ideal, $L(\lambda'')S = L(\lambda')$. This will give us the two desired Borel-fixed ideals.
\end{proof}

Note that \autoref{poly1} corresponds to cases \textit{(i)} - \textit{(v)} in \autoref{THEOREM_A} while \autoref{poly2} corresponds to the other cases. 

\begin{remark} Computing Hilbert polynomials of strongly stable ideals by comparing it to the lexicographic ideal is a very useful technique. See, for instance, \cite{MALL, LELLA_12, KAMBE_LELLA}.
\end{remark}



\section{Deformation Theory} \label{DeformationSection} 
In this section we compute the tangent space to the non lexicographic Borel-fixed ideal, $[I(\lambda)]$, and provide a partial basis for the second cotangent cohomology group of $S/I(\lambda)$. These are essential for the computation of the universal deformation space of $I(\lambda)$, which we carry out in \autoref{PROOF}. The general procedure to compute the universal deformation space can be found in \cite[\S 3]{js} and \cite[\S 5]{ps}. We begin with a useful result that relates the universal deformation space of an ideal $I$ to an analytic neighbourhood of $[I]$ in its Hilbert scheme.

\begin{theorem}[Comparison Theorem \cite{ps}] \label{compare} Let $X\subseteq \PP^{n}$ be a subscheme with ideal $I_X=(f_{1},\dots,f_{s})$ where $\deg f_{i}=d_{i}$ satisfying, $(\kk[x_{0},\dots,x_{n}]/I_X)_{e}\simeq H^{0}(\mathcal{O}_{X}(e))$ for $e=d_{1},\dots,d_{s}$. Then there is an isomorphism between the universal deformation space of $I_X$ and that of $X$; the latter is an analytic neighbourhood of $\Hilb(\PP^n)$ around $[X]$. In particular, 
$$T_{[I_X]}\, \Hilb(\PP^n) = H^0(\PP^n,\mathcal{N}_{X/\PP^n}) = \Hom(I_X,S/I_X)_0.$$
\end{theorem}

From \autoref{poly1} and  \autoref{poly2} we see that $I(\lambda)$ lies inside a unique $\mathbf{P}^{d+2}$. As a consequence, any embedded deformation of the $I(\lambda)$ in $\PP^n$ can be realized as a deformation of the $I(\lambda)$ \textit{in} $\mathbf{P}^{d+2}$ along with a deformation of $\mathbf{P}^{d+2}$ in $\mathbf{P}^n$. In other words, \'etale locally around $[I(\lambda)]$ we have an isomorphism 
\begin{equation} \label{SPLIT}
\Hilb^{P_\lambda}(\PP^n) \simeq  \Hilb^{P_\lambda}(\PP^{d+2}) \times \mathbf{A}^{(d+3)(n-d-2)}.
\end{equation}
As a consequence, it suffices to consider the case $n=d-2$. 

\begin{notation}  For the rest of this section we assume $n=d-2$. We also assume  $\lambda$ is of the form $((d+1)^q,1)$ satisfying the conditions of \autoref{poly1}, or of the form $((d+1)^q,r+1,1)$ satisfying the conditions of \autoref{poly2}. In the first case the corresponding non lexicographic ideal is
$$
I(\lambda) = x_0(x_0,\dots,x_{n-1}) + (x_1^{q})
$$
and in the second case it is
$$
I(\lambda) = x_0(x_0,\dots,x_{n-1}) + x_1^q(x_1,\dots,x_{n-r-1}).
$$
\end{notation}

We start by verifying that the comparison theorem holds in all cases of interest. 

\begin{lemma} \label{comparison} If $\lambda \ne (1^4)$ then $(S/I(\lambda))_{e}\simeq H^{0}(\PP^n,\mathcal{O}_{\Proj(S/I(\lambda))}(e))$ for all $e \geq 1$. 
\end{lemma}
\begin{proof} For the purpose of this proof it will be convenient to unify notation and express 
$$I(\lambda) = x_0(x_0,\dots,x_{n-1}) + x_1^q(x_1,\dots,x_{p})$$
with $0 \leq p \leq n-1$.  Let  $X = \Proj(S/I(\lambda))$ and assume $p \ne n-1$.  Let $J = (x_0) + x_1^q(x_1,\dots,x_{p})$ and consider the exact sequence $0 \longrightarrow J/I(\lambda) \longrightarrow S/I(\lambda) \longrightarrow S/J \longrightarrow 0$. The associated long exact sequence in local cohomology of graded $S$-modules is
\begin{equation*}
0 \longrightarrow 
	H_{\mm}^0(J/I(\lambda)) \longrightarrow 
	H_{\mm}^0(S/I(\lambda)) \longrightarrow  
	H_{\mm}^0(S/J) \longrightarrow 
	H_{\mm}^{1}(J/I(\lambda)) \longrightarrow  
	H_{\mm}^{1}(S/I(\lambda)) \longrightarrow  
	H_{\mm}^{1}(S/J).
\end{equation*}
Since $x_{n-1}$ and $ x_n$ are nonzero divisors on $S/J$ we have  $\mathrm{depth}_{\mm}(S/J) \geq 2$. This implies that the local cohomology groups $H^0_{\mm}(S/J)$ and $H^1_{\mm}(S/J)$ are zero. As graded $S$-modules, we have $J/I(\lambda) \simeq (S/(x_0,\dots,x_{n-1}))(-1):= \bar{S}(-1)$. The associated sheaf on $\PP^n$ is just the structure sheaf of a point. Consider the following exact sequence
\begin{equation*}
0 \longrightarrow H_{\mm}^0(\bar{S}(-1)) \longrightarrow \bar{S}(-1) \longrightarrow H^0_{\star}(\mathcal{O}_{\mathrm{pt}}(-1)) \longrightarrow H^1_{\mm}(\bar{S}(-1)) \longrightarrow 0.
\end{equation*}
For all $e \geq 1$ we have $H^0_{\star}(\mathcal{O}_{\mathrm{pt}}(-1))_e = H^0(\mathcal{O}_{\mathrm{pt}}(e-1)) = H^0(\mathcal{O}_{\mathrm{pt}}) = \kk \simeq \bar{S}(-1)_{e}$. Thus, we have $H^0_{\mm}(\bar{S}(-1))_e = H^1_{\mm}(\bar{S}(-1))_e = 0$ for all $e \geq 1$.

Combining this with the first long exact sequence we obtain $H^0_{\mm}(S/I(\lambda))_e = H^1_{\mm}(S/I(\lambda))_e = 0$ for all $e \geq 1$. The desired result now follows from using the exact sequence
\begin{equation*}
0 \longrightarrow H_{\mm}^0(S/I(\lambda)) \longrightarrow S/I(\lambda) \longrightarrow H^0_{\star}(\PP^n,\mathcal{O}_{X}) \longrightarrow H^1_{\mm}(S/I(\lambda)) \longrightarrow 0.
\end{equation*}

The remaining case is when $p=n-1$ and $q=1$ (we excluded the case of $n=2, q =2$). In this case the regularity of $I(\lambda)$ is $2$ \cite[Corollary 3.1]{sp}. Thus Corollary 4.8 and Proposition 4.16 in \cite{syzygies} establish that $\dim_{\kk} (S/I(\lambda))_e = P_{S/I(\lambda)}(e) = P_{X}(e) = h^0(\PP^n,\mathcal{O}_X(e))$ for all $e\geq 1$.
\end{proof}

The next four propositions provide a basis for the tangent space to each $[I(\lambda)]$. Since their proofs are very similar we will only provide all the details for the first one.

\begin{definition} For $S = \kk[x_0,\dots,x_n]$ and for $q \geq 1$ define the following subsets
\begin{enumerate}
\item $\mathcal{T}_1 = \{x_{i_1}\cdots x_{i_{q}}: 1 \leq i_1 \leq i_2 \leq \cdots \leq i_{q} \leq n  \} \, \setminus \, \{x_1^q,x_1^{q-1}x_2,\dots,x_1^{q-1}x_n\}.$
\item $\mathcal{T}_2 = \{x_1^{q-1}x_2,\dots,x_1^{q-1}x_n\}$.
\end{enumerate}
\end{definition}

\begin{proposition} \label{tangent1} Let $\lambda = ((n-1)^q,r+1,1)$ be an integer partition. Assume $n \geq 4$ and either $r \geq 2$ and $q\geq 1$, or $r=1$ and $q \geq 3$. Then
\begin{equation*}
\dim_\kk T_{[I(\lambda)]} \, \Hilb^{P_\lambda}(\PP^n) =  3n-1 + (n-r-2)(r+1) + \binom{n+q-1}{n-1}.
\end{equation*}
A general $\varphi \in \Hom(I(\lambda),S/I(\lambda))_{0}$ can be written as
\begin{eqnarray*}
\varphi(x_0^2) & = & a_0x_0x_n \\
\varphi (x_0x_i) & = & a_{i}x_0x_n  + c_1x_1x_i + c_2x_2x_i + \cdots + c_nx_nx_{i}, \quad 1 \leq i \leq n-1\\
\varphi(x_1^{q+1}) & = & b_{1}x_0x_{n}^q + \sum_{\omega \in \mathcal{T}_1}c_{\omega}x_1\omega +  \ell^{1}_{n-r}x_1^qx_{n-r} + \cdots \ell^{1}_{n}x_1^qx_{n} , \quad \, 1 \leq i \leq n-r-1 \\
\varphi(x_1^qx_i) & = & b_{i}x_0x_{n}^q + \sum_{\omega \in \mathcal{T}_1 \cup \mathcal{T}_2}c_{\omega}x_i\omega +  \ell^{i}_{n-r}x_1^qx_{n-r} + \cdots \ell^{i}_{n}x_1^qx_{n} , \quad \, 2 \leq i \leq n-r-1
\end{eqnarray*}
where $a_0,\dots,a_{n-1}, b_1,\dots,b_{n-r-1},c_1,\dots,c_n$, $\{c_{\omega}\}_{\omega \in \mathcal{T}_1 \cup \mathcal{T}_2}$, and $\{\ell_{j}^{i}\}_{n-r \leq j \leq n}^{1\leq i \leq n-r-1}$ are independent parameters.
\end{proposition}
\begin{proof} By \autoref{compare} and  \autoref{comparison}, $\dim_\kk T_{[I(\lambda)]} \, \Hilb^{P_\lambda}(\PP^n) = \dim_\kk \Hom(I(\lambda),S/I(\lambda))_0$. 
Let $F_1 \overset{\psi_1}{\longrightarrow} F_0 \overset{\psi_0}{\longrightarrow} I(\lambda) \longrightarrow 0$ be the beginning of the Eliahou-Kervaire resolution from  \autoref{Eliahou-Kervaire}.  We have the following exact sequence
\begin{equation*}
0 \longrightarrow \Hom(I(\lambda),S/I(\lambda))_0 \longrightarrow \Hom(F_0,S/I(\lambda))_0 \overset{\psi_1^{\vee}}{\longrightarrow} \Hom(F_1,S/I(\lambda))_0.
\end{equation*}
Dualizing $\psi_1$ we see that $\phi \in \Hom(I(\lambda),S/I(\lambda))_0$ if and only if the following relations hold in $S/I(\lambda)$
\begin{eqnarray*}
	\phi(x_{0}x_{i})x_{j} & =& \phi(x_{0}x_{j})x_{i},\quad 0 \leq i,j \leq n-1 \\
	\phi(x_{0}x_j)x_{1}^{q} & = & \phi(x_{1}^{q}x_{j})x_{0},\quad 1 \leq j \leq n-r-1 \\
	\phi(x_{1}^qx_{i})x_{j} & = & \phi(x_{1}^qx_{j})x_{i},\quad 1 \leq i,j \leq n-r-1.
\end{eqnarray*}

It is straightforward to check that the family described in the statement satisfies these relations. 

Conversely, given $\phi \in \Hom(I(\lambda),S/I(\lambda))_0$ we need to show that $\phi$ lies in our family. For any $i \ne n-1$, the relation $\phi(x_{0}x_{i})x_{n-1}=\phi(x_{0}x_{n-1})x_{i}$ implies that $x_i$ divides all the monomials in the support of $\phi(x_0x_i)$ that are not annihilated by $x_{n-1}$. But the only quadratic monomial that is non-zero in $S/I$ and annihilated by $x_{n-1}$ is $x_0x_{n}$. Thus, for $ i \neq n-1$ the image $\phi(x_{0}x_{i})$ is supported on $\{x_1x_i,x_{2}x_{i},\dots,x_{n}x_{i},x_{0}x_{n}\}$. Since $r \geq 2$ or $q \geq 3$, the only quadratic monomial (non-zero in $S/I(\lambda)$) annihilated by $x_{n-2}$ is $x_{0}x_n$. Thus the relation $\phi(x_0x_{n-2})x_{n-1} = \phi(x_0x_{n-1})x_{n-2}$ implies  $\phi(x_{0}x_{n-1})$ is also supported on $\{x_1x_{n-1},x_{2}x_{n-1},\dots,x_{n}x_{n-1},x_{0}x_{n}\}$. Analogously, we may use the relation $\phi(x_{1}^qx_{i})x_{j} = \phi(x_{1}^qx_{j})x_{i}$ to deduce that $\phi(x_{1}^qx_{i})$ is supported on $\{x_1^qx_{n-r}\dots,x_1^qx_{n}\,x_{0}x_{n}^q\} \cup x_i\mathcal{T}_1 \cup x_i\mathcal{T}_2$.

Let $\phi(x_{0}x_{n-1})=a_{n-1}x_{0}x_{n}+c_{2}x_{n-1}x_{2}+\cdots c_{n}x_{n-1}x_{n}$ for some constants $c_i$. Then for $j \ne n-1$, the relation $x_{j}\phi(x_{0}x_{n-1})=x_{n-1}\phi(x_{0}x_{j})$ implies $\phi(x_{0}x_{j})=a_{j}x_{0}x_{n}+c_{2}x_{j}x_{2}+\cdots +c_{n}x_{j}x_{n}$ for some constant $a_j$. Now assume 
\begin{equation*}
\phi(x_{1}^{q}x_2) = b_2x_0x_n^q+ \sum_{\omega \in \mathcal{T}_1 \cup \mathcal{T}_2}c_{\omega}x_2\omega + \ell^{2}_{n-r}x_1^qx_{n-r} + \cdots + \ell^{2}_{n}x_1^qx_n.
\end{equation*}
with $c_{\omega}, \ell^{2}_i,b_2$ some constants. For $j  \geq 3$ the relation $\phi(x_{1}^{q}x_2)x_{j} = \phi(x_{1}^qx_{j})x_{2}$ implies
\begin{equation*}
\phi(x_{1}^{q}x_j) = b_jx_0x_n^q+ \sum_{\omega \in \mathcal{T}_1 \cup \mathcal{T}_2}c_{\omega}x_j\omega + \ell^{j}_{n-r}x_1^qx_{n-r} + \cdots + \ell^{j}_{n}x_1^qx_n.
\end{equation*}
where $l^{j}_{i},b_j$ are constants. Note that if $j=1$ then the non-zero elements of $x_j\mathcal{T}_2$ are  $\{x_1^qx_{n-r},\dots,x_1^qx_{n}\}$. Thus, $\phi(x_1^{q+1})$ is also of the desired form and this completes the proof.
\end{proof}

\begin{proposition} \label{tangent2} Let $\lambda = (n-1,2,1)$ be an integer partition with $n 
\geq 4$. Then 
\begin{equation*}
\dim_\kk T_{[I(\lambda)]} \, \Hilb^{P_\lambda}(\PP^n) = 6n-6.
\end{equation*}
A general $\varphi \in \Hom(I(\lambda),S/I(\lambda))_{0}$ can be written as
\begin{eqnarray*}
	\varphi(x_{0}^{2}) & = & a_{0}x_{0}x_{n}\\
	\varphi(x_{0}x_{i}) & = & a_{i}x_{0}x_{n}+c_{2}x_{2}x_{i}+c_{3}x_{3}x_{i}+\cdots+c_{n}x_{n}x_{i}, \quad 1 \leq i \leq n-2 \\
	\varphi(x_{0}x_{n-1}) & = & a_{n-1}x_{0}x_{n} + c_{1}x_{1}x_{n-1}+c_{2}x_{2}x_{n-1}+\cdots+c_{n}x_{n}x_{n-1} + \boxed{\alpha x_1x_n} \\
	\varphi(x_{1}^{2}) & = & b_{1}x_{0}x_{n}+ \ell_{n-1}^{1}x_{1}x_{n-1}+\ell_{n}^{1}x_{1}x_{n}\\
	\varphi(x_{1}x_{i}) & = & b_{i}x_{0}x_{n}+d_{2}x_{2}x_{i}+\cdots+d_{n}x_{n}x_{i}+\ell_{n-1}^{i}x_{1}x_{n-1} + \ell_{n}^{i}x_{1}x_{n}, \quad 2 \leq i \leq n-r-1.
\end{eqnarray*}
where $\alpha,a_0,\dots,a_{n-1}, b_1,\dots,b_{n-2},c_1,\dots,c_n, 
d_2,\dots,d_n$ and $\{\ell_{n-1}^{i},\ell_n^i\}_{1\leq i \leq n-2}$ are independent parameters.
\end{proposition}

\begin{proposition} \label{tangent3} Let $\lambda = (n-1,1,1)$ be an integer partition with $n \geq 3$. Then
\begin{equation*}
\dim_\kk T_{[I(\lambda)]} \, \Hilb^{P_\lambda}(\PP^n) = 6n-4.
\end{equation*}
A general $\varphi \in \Hom(I(\lambda),S/I(\lambda))_{0}$ can be written as
\begin{eqnarray*}
	\varphi(x_{0}^{2}) & = &  a^0_{0}x_{0}x_{n} + a^1_{0}x_1x_n\\
	\varphi(x_{0}x_{1}) & = & a^0_{1}x_{0}x_{n}+a^1_{1}x_1x_n\\
	\varphi(x_{0}x_{i}) & = & a^0_{i}x_{0}x_{n}+a^1_{i}x_1x_n+c_{2}x_{2}x_{i}+c_{3}x_{3}x_{i}+\cdots+c_{n}x_{n}x_{i}, \quad 2 \leq i \leq n-1 \\
	\varphi(x_{1}^{2}) & = & b^0_{1}x_{0}x_{n}+b^1_{1}x_1x_n\\
	\varphi(x_{1}x_{i}) & = & b^0_{i}x_{0}x_{n}+b^1_{i}x_1x_n+d_{2}x_{2}x_{i}+\cdots+d_{n}x_{n}x_{i}, \quad 2 \leq i \leq n-1.
\end{eqnarray*}
where $c_2,\dots,c_n,d_2,\dots,d_n,\{a^0_i,a^1_i\}_{0 \leq i \leq n-1}, \{b^0_i,b^1_i\}_{1 \leq i \leq n-1}$ are independent parameters.
\end{proposition}

\begin{proposition} \label{tangent4} 
Let $\lambda = ((n-1)^q,1)$ be an integer partition where either $n=3$ and $q \geq 4$, or $n \geq 4$ and $q \geq2$. Then
\begin{equation*}
\dim_\kk T_{[I(\lambda)]} \, \Hilb^{P_\lambda}(\PP^n) = 2n-1 + \binom{n+q-1}{n-1}.
\end{equation*}
A general $\varphi \in \Hom(I(\lambda),S/I(\lambda))_{0}$ can be written as
\begin{eqnarray*}
\varphi(x_0^2) & = & a_0x_0x_n \\
\varphi (x_0x_i) & = & a_ix_0x_n  + c_1x_1x_i + \cdots + c_nx_nx_i \\
\varphi(x_1^{q}) & = & b_{1}x_0x_{n}^{q-1} + \sum_{\omega \in \mathcal{T}_1 \cup \mathcal{T}_2 \setminus x_{n}^q}c_{i,\omega}\omega,
\end{eqnarray*}
where $a_0,\dots,a_{n-1}, b_1,c_1,\dots,c_n,c_{i,\omega}$ are independent parameters.
\end{proposition}

As we will see in \autoref{PROOF}, for $\lambda = ((n-1)^q,1)$ the ideal $I(\lambda)$ corresponds to a smooth point on its Hilbert scheme. To understand the geometry in a neighborhood of the other $[I(\lambda)]$, we will need to compute its deformation space. To do this, we may exclude the trivial deformations, those induced by coordinate changes, as they are unobstructed. More precisely, we want to compute $T^1(R/\kk,R)_0$ where $R =S/I(\lambda)$ \cite[\S 3, p. 24]{js}. A straightforward computation of the partial derivatives gives the following bases for $T^1$ . 

\begin{corollary} \label{nontrivial1} Let $\lambda = ((n-1)^q,r+1,1)$ be an integer partition and let $R = S/I(\lambda)$. Assume $n \geq 4$ and either $r \geq 2$ and $q\geq 1$, or $r=1$ and $q \geq 3$. Then $T^1(R/\kk,R)_0$ is spanned by
\begin{align*}
\varphi (x_0x_i) & =  a_ix_0x_n,  \qquad 0 \leq i \leq n-r-1\\
\varphi(x_0x_i) & =  0,  \qquad  n-r \leq i \leq n-1 \\
\varphi(x_1^{q+1}) & =  b_{1}x_0x_{n}^q + \sum_{\omega \in \mathcal{T}_1}c_{\omega}x_1\omega +  \ell^{1}_{n-r}x_1^qx_{n-r} + \cdots \ell^{1}_{n}x_1^qx_{n} \\
\varphi(x_1^qx_i) & =  b_{i}x_0x_{n}^q + \sum_{\omega \in \mathcal{T}_1}c_{\omega}x_i\omega, \qquad 1 \leq i \leq n-r-1,
\end{align*}
where $a_0,\dots,a_{n-1}, b_1,\dots,b_{n-r-1},\ell^{1}_{n-r},\dots,\ell^{1}_n$ and $\{c_{\omega}\}_{\omega \in \mathcal{T}_1}$ are independent parameters.
\end{corollary}

\begin{corollary} \label{nontrivial2} Let $\lambda = (n-1,2,1)$ be an integer partition with $n 
\geq 4$ and let $R = S/I(\lambda)$. Then $T^1(R/\kk,R)_0$ is spanned by
\begin{align*}
	\varphi(x_{0}x_{i}) & =  a_{i}x_{0}x_{n}, \qquad 0 \leq i \leq n-2 \\
	\varphi(x_{0}x_{n-1}) & =  \alpha x_1x_n\\
	\varphi(x_{1}^2) & =  b_{1}x_{0}x_{n} +d_{n-1}x_{1}x_{n-1}+d_{n}x_{1}x_{n} \\
	\varphi(x_{1}x_{i}) & =  b_{i}x_{0}x_{n}, \qquad 2 \leq i \leq n-r-1,		
\end{align*}
where $\alpha,a_0,\dots,a_{n-2}, b_1,\dots,b_{n-2},d_{n-1},d_n$ 
are independent parameters.
\end{corollary}

\begin{corollary} \label{nontrivial3} Let $\lambda = (n-1,1,1)$ be an integer partition with $n \geq 3$ and let $R = S/I(\lambda)$. Then $T^1(R/\kk,R)_0$ is spanned by
\begin{align*}
	\varphi(x_{0}x_{i}) & =  a^0_{i}x_{0}x_{n}+a^1_{i}x_1x_n, \quad 0 \leq i \leq n-1 \\
	\varphi(x_{1}^2) & =  b^0_{1}x_{0}x_{n}+b^1_{1}x_1x_n, \quad 0 \leq i \leq n-1 \\
	\varphi(x_{1}x_{i}) & =  b^0_{i}x_{0}x_{n}, \, \quad \quad \quad \quad \quad 2 \leq i \leq n-1,
\end{align*}
where $a^0_i,a^1_i,b^0_{i}$ are independent parameters.
\end{corollary}


\begin{lemma} \label{T2} With notation as in \autoref{Eliahou-Kervaire}, let $F$ denote the Eliahou-Kervaire resolution of $I(\lambda)$. Let $R = S/I(\lambda)$ and let $\f_{li}^j \in \Hom(F_1,R)$ denote the dual of $\e_{li}^j$.
\begin{enumerate} 
\item[i)] If $\lambda = ((n-1)^q,r+1,1)$ then $\{x_0x_n^2\f_{0i}^j,x_0x_n^{q+1}\f_{1i}^{j}\}_{i,j} \subseteq T^2(R/\kk,R)_0$ is linearly independent.
\item[ii)] If $\lambda = (n-1,2,1)$ then $\{x_0x_n^{2}\f_{0i}^j\,x_0x_n^{2}\f_{1i}^j,x_1x_n^2\f_{0,n-1}^j\}_{i,j} \subseteq T^2(R/\kk,R)_0$ is linearly independent.
\item[iii)] If $\lambda = (n-1,1,1)$ then $\{x_0x_n^2\f_{0i}^j\,x_0x_n^2\f_{1i}^j,x_1x_n^2\f_{0i}^j\,x_1x_n^2\f_{1i}^j\}_{i,j} \subseteq T^2(R/\kk,R)_0$ is linearly independent.
\end{enumerate}
\end{lemma}
\begin{proof} We will only prove \textit{ii)} as the other two cases are analogous (and simpler). We use $A_i$ to denote the matrix associated to $\psi_i$. By construction the entries in $A_i$ are supported on $(x_0,\dots,x_{n-1})$. 
Dualizing the resolution $F$ we obtain
\begin{align*}
\psi_1^{\vee}(\f_{00}^{\star}) & =  - x_1\f_{01}^{0} -  \sum_{1<j \leq n-1} x_j\f_{0j}^{0} \\
\psi_1^{\vee}(\f_{01}^{\star}) &= x_0\f_{01}^{0}  - x_1\f_{11}^0 -\sum_{1 < j \leq n-1}x_j\f_{0j}^1 \\
\psi_1^{\vee}(\f_{0i}^{\star})& = x_0\f_{0i}^0 + x_1\f_{0i}^1 - x_1\f_{1i}^0 + \sum_{2 \leq j < i}x_j\f_{0i}^{j} - \sum_{i < 
j \leq n-1}x_j\f_{0j}^i \\
\psi_1^{\vee}(\f_{0,n-1}^{\star}) &= x_0\f_{0,n-1}^0  + x_1\f_{0,n-1}^1 + \sum_{2 \leq j < n-1}x_j\f_{0,n-1}^{j} \\
\psi_1^{\vee}(\f_{1i}^{\star}) &=  x_0\f_{0j}^{0} + \sum_{1 \leq j < i}x_j\f_{1i}^{j} - \sum_{i < j \leq n-2}x_j\f_{1j}^i.
\end{align*} 
Let us first check that $x_0x_n^2\f_{0i}^j$ and $x_0x_n^2\f_{1i}^j$ are well defined elements of $T^2(R/\kk,R)_0$. It is enough to show that $x_0x_n^2$ annihilates $\ker \psi_1 + \text{Kos}$. Since the entries in $A_2$ are supported on $(x_0,\dots,x_{n-1})$, multiplying by $x_0x_n^2$ annihilates $\psi_2(F_2) = \ker \psi_1$. Since the Koszul relations are supported on $(x_0,x_1)$, $x_0x_n^2$ annihilate $\text{Kos}$.

Since $x_1x_n^2$ also annihilates $\text{Kos}$, to show that that $x_1x_n^2\f_{0,n-1}^{j}$ is a well defined element, we only need to prove that $x_1x_n^2$ annihilates the restriction $(\ker \psi_1)|_{S(-3)\e^{j}_{0,n-1}}$. Let $\boldsymbol{v} \in \ker \psi_1$ 
and since the differentials are linear we may assume $\boldsymbol{v}$ is linear. Then $\psi_1(\boldsymbol{v})=0$ implies
\begin{equation*}
\begin{aligned}
-x_1\boldsymbol{v}_{\e_{01}^{0}} - x_2\boldsymbol{v}_{\e_{02}^{0}} - \cdots - x_{n-1}\boldsymbol{v}_{\e_{0,n-1}^{0}} = 0& \\
x_0\boldsymbol{v}_{\e_{01}^{0}}-x_1\boldsymbol{v}_{\e_{11}^{0}} - x_2\boldsymbol{v}_{\e_{02}^{1}} - \cdots - x_{n-1}\boldsymbol{v}_{\e_{0,n-1}^{1}} = 0& \\
x_0\boldsymbol{v}_{\e_{0i}^0} + x_1\boldsymbol{v}_{\e_{0i}^1} - x_1\boldsymbol{v}_{\e_{1i}^0} + \sum_{2 \leq j < i}x_j\boldsymbol{v}_{\e_{0i}^{j}} - \sum_{i < j \leq n-1}x_j\boldsymbol{v}_{\e_{0j}^i} = 0&, \quad 2 \leq i \leq n-2.
\end{aligned}
\end{equation*}
The $j$-th equation above is just the 
$j$-th row of $A_1$ multiplied with $\boldsymbol{v}$ (we can read this off from our description of $\psi_1^{\vee}$). From the $j$-th equation we can see that $\boldsymbol{v}_{\e_{0,n-1}^{j}}$ is supported on $(x_0,\dots,x_{n-2})$ for all $0 \leq j \leq n-2$. As a consequence, $x_1x_n^2$ annihilates $\boldsymbol{v}_{\e_{0,n-1}^{j}}$ and all of $(\ker \psi_1)|_{S(-3)\e^{j}_{0,n-1}}$.

We will now show that the set $\mathcal{S} = \mathrm{span}_\kk \{x_0x_n^{2}\f_{0i}^j\,x_0x_n^{2}\f_{1i}^j,x_1x_n^2\f_{0,n-1}^{j}\}_{i,j}$ is linearly independent in $T^2(R/\kk,R)$. In particular, we need to show that no non-zero element of $\mathcal{S}$ is a linear combination of the form $\sum_{l,i} c_{li}Q_{li}\overline{\psi_{1}^{\vee}}(\f_{li}^{\star})$ where $Q_{li} \in R(2)$ are quadrics and $c_{li} \in \kk$ constants. However, since all the elements of $\mathcal{S}$ are multiples of $x_n^2$ and $A_1$ does not contain the variable $x_n$, it suffices to show that no non-zero element of $\mathcal{S}$ is a linear combination of the form $\sum_{l,i} c_{li}x_n^2\overline{\psi_{1}^{\vee}}(\f_{li}^{\star})$. From the description of $\psi_1^{\vee}$ in the first paragraph we see that this is indeed the case.
\end{proof}

\section{Proof of the main theorem} \label{PROOF}

The goal of this section is to prove the main theorem of this paper. The proof will provide a description of the universal deformation space of $I(\lambda)$ valid in all characteristics.

\begin{customthm}{A} \label{THEOREM_A}
Let $\text{char}(\kk) = 0$. A Hilbert scheme  $\Hilb^{P_\lambda}(\PP^n)$ with two Borel-fixed points is either smooth, irreducible and singular, or a union of two components.
\begin{itemize}[topsep=0pt,itemsep=1ex]
\item It is smooth when 
\begin{enumerate}
\item[(i)] $\lambda = (n^s,1,1,1)$ for $n \geq 2$. For $s=0$ its  general member parameterizes three isolated points.
\item[(ii)] $\lambda = (2^s,1,1,1,1)$. For $s = 0$ its general member parameterizes four isolated points in the plane.
\item[(iii)]  $\lambda = (n^s,2^q,1)$ with and $n > 2$ and $q \geq 4$: For $s=0$ its general member parameterizes a plane curve of degree $q$ union an isolated point.
\item[(iv)] $\lambda = (n^s,(d+1)^q,1)$ with $n > d+1 > 2$ and $q \geq 2$. For $s=0$ its general member parameterizes a hypersurface of degree $q$ in a $\PP^{d+1}$ union an isolated point.
\end{enumerate}
\item It is irreducible with normal, Cohen-Macaulay singularities when
\begin{enumerate}
\item[(v)] $\lambda=(n^s,(d+1)^q,r+1,1)$ with $n > d+1  > r +1 > 2$. For $s=0$ its general member  parameterizes a hypersurface of degree $q$ in a $\PP^{d+1}$ union a $r$-plane inside $\PP^{d+1}$ and an isolated point; the hypersurface meets the $r$-plane transversely in $\PP^{d+1}$. If $d= n-2$, the Hilbert scheme at the non lexicographic point is \'etale-locally a cone over the Segre embedding $\mathbf{P}^1 \times \mathbf{P}^{n-r-1} \hookrightarrow \PP^{2(n-r)-1}$.
\item[(vi)]   $\lambda=(n^s,(d+1)^q,2,1)$ with $n > d+1  > 2$ and $q \geq 3$. The description of its general member is identical to case \textit{(v)}.
\item[(vii)] $\lambda = (n^s,d+1,1,1)$ with $n > d+1 > 1$. For $s=0$ the general member parameterizes a $d$-plane union two isolated points. If $d = n-2$ the Hilbert scheme at the non lexicographic point is \'etale-locally a cone over the Segre embedding $\mathbf{P}^2 \times \mathbf{P}^{n-1} \hookrightarrow \PP^{3n-1}$. In particular, if $n=3$ it parameterizes a line union two isolated points, is Gorenstein.
\end{enumerate}
\item It is a union of two components when
\begin{enumerate}
\item[(viii)] $\lambda = (n^s,2,2,1)$ for $n \geq 3$. The Hilbert scheme $\Hilb^{P_\lambda}(\PP^n)$ is a union of two smooth irreducible components meeting transversely. When $s=0$, the general member of one component parameterizes a plane conic union an isolated point and the general member of the other component parameterizes two skew lines.
\item[(ix)] $\lambda = (n^s,d+1,2,1)$ with $n > d+1 > 3$. The Hilbert scheme $\Hilb^{P_\lambda}(\PP^n)$ is reduced with two irreducible components $\mathcal{Y}_1$ and $\mathcal{Y}_2$. 
\begin{itemize}
\item When $s=0$ the component $\mathcal{Y}_1$ is smooth and its general member parameterizes a disjoint union of a $d$-plane union a line. If $d=n-2$ the component is isomorphic to the blowup of $\mathbf{G}(1,n) \times \mathbf{G}(n-2,n)$ along the locus $\{(L,\Lambda): L \subseteq \Lambda \}$. 
\item When $s = 0$ the component $\mathcal{Y}_2$ is normal and Cohen-Macaulay. Its general point parameterizes a $d$-plane union a line and an isolated point; the $d$-plane meets the line at a point. If $d=n-2$ the component at the non lexicographic point is \'etale-locally a cone over the Segre embedding $\mathbf{P}^1 \times \mathbf{P}^{n-2} \hookrightarrow \PP^{2(n-1)-1}$.
\end{itemize}
\end{enumerate}
\end{itemize} 
\end{customthm}

By \autoref{DETACH} and Equation \autoref{SPLIT} we may assume that $s =0$ and $n=d-2$.

\begin{proof}[Proof of \autoref{THEOREM_A} (i), (ii) and (viii)] Case \textit{(i)} and \textit{(ii)} are \cite[Theorem 2.4]{FOGARTY} while case \textit{(viii)} is \cite[Theorem 1.1]{ccn}.
\end{proof}

\begin{proof}[Proof of \autoref{THEOREM_A} (iii), (iv)] It follows from \cite[Theorem 4.1]{rs} that $\dim( \Hilb^{P_\lambda}(\PP^n))$ agrees with the dimension of the tangent space to $[I(\lambda)]$ (\autoref{tangent4}). Thus, $[I(\lambda)]$ is a smooth point on the Hilbert scheme. By Theorem \cite[Theorem 1.4]{rs} the lexicographic point is also a smooth point. Since $\Hilb^{P_\lambda}(\PP^n)$ has only two Borel-fixed points (\autoref{poly1}), \autoref{borel2} implies that the Hilbert scheme is smooth. Finally, \cite[Theorem 4.1]{rs} gives the description of the general member.
\end{proof}

\begin{proof}[Proof of \autoref{THEOREM_A} (v), (vi)] \label{fivesix} Let $\mathbf{U} = \kk[\![u_{00},\dots,u_{0,n-r-1},u_{11},\dots,u_{1n},\{u_{2,\omega}\}_{\omega \in \mathcal{T}_1}]\!]$ and let $\mathfrak{m}_{\mathbf{U}}$ denote its maximal ideal. Consider the following perturbation of $\psi_0$
\begin{align*}
\Psi_0(\e_{0i}^{\star}) & =  x_0x_i +u_{0i}x_0x_n, \quad i \leq n-r-1 \\
\Psi_0(\e_{0i}^{\star}) &= x_0x_i, \quad i \geq n-r \\  
\Psi_0(\e_{11}^{\star}) &=   x_1^{q+1} +u_{11}x_0x_n^{q} + \sum_{l=0}^{r}u_{1,n-r+l}x_1^qx_{n-r+l} 
+ \sum_{\omega \in \mathcal{T}_1} u_{2,\omega}x_1 \omega
+ \sum_{l=0}^r\sum_{\omega \in \mathcal{T}_1} u_{1,n-r+l}u_{2,\omega}x_{n-r+l}{\omega} \\
\Psi_0(\e_{1i}^{\star}) &= x_1^{q}x_i+u_{1i}x_0x_n^q + 
				\sum_{\omega \in \mathcal{T}_1} u_{2,\omega}x_i \omega,  \quad i > 1.
\end{align*}
By \autoref{nontrivial1} this lifts the first order deformation by non-trivial deformations. To perturb the syzygies, we need a few definitions. Let $\mathcal{U} := \{\omega \in \mathcal{T}_1: \text{there exists } x_i | \omega \text{ with } n-r \leq i \leq n-1\} $, $\mathcal{V} := \{\omega \in \mathcal{T}_1: \omega \text{ is supported on } x_{1},\dots,x_{n-r-1},x_n\} \setminus x_n^q$ and $\eta := x_n^q$.
Observe that  $\mathcal{T}_1 = \mathcal{U} \sqcup \mathcal{V} \sqcup \{x_n^q\}$. 

For each $\omega \in \mathcal{U}$ choose some $n-r \leq i \leq n-1$ for which $x_i | \omega$ and let $\bar{\omega} := \frac{\omega}{x_i}$ and $\widehat{\omega} := i$.  For each $\omega \in \mathcal{V}$ define the following
\begin{itemize}
\item  Let $\omega_0 = 1$ and for $1 \leq \ell \leq q $ let $\omega_{\ell}$ denote the lexicographically largest monomial of degree $\ell$ dividing $\omega$.
\item For  $0 \leq  \ell \leq q-1$ let $\lambda(\omega_{\ell})$ to be the index of the variable $\frac{\omega_{\ell+1}}{\omega_{\ell}}$.
\item For $0 \leq \ell \leq q-1$ let $u_{\omega_{\ell}} := \frac{\omega}{\omega_{\ell}} |_{\{x_j = u_{0j}\}_j}$.
\end{itemize}
For example, if $\omega = x_0^3x_3^3x_4$ then $\omega_{4} = x_0^3x_3$, then $\lambda(\omega_{3}) = x_3$ and $u_{\omega_{4}} = u_{03}^2u_{04}$. 
Define
\begin{equation*}
\Omega := \sum_{\ell=1}^{q}(-1)^{\ell-1}u_{01}^{\ell-1}x_1^{q-\ell}x_n^{\ell}\e^{\star}_{01}  
+ \sum_{\omega \in \mathcal{U}}u_{2\omega}\bar{\omega} x_n\e^{\star}_{0,\widehat{\omega}}
+ \sum_{\omega \in \mathcal{V}}u_{2\omega}\sum_{\ell=1}^{q}(-1)^{\ell-1}u_{\omega_{q-\ell+1}}\omega_{q-\ell}x_n^{\ell}\e_{0,\lambda(\omega_{q-\ell})}^{\star}.
\end{equation*}
Here is the lift of the syzygies
\begin{align*}
\Psi_1(\e_{0i}^{j}) & =  (x_j+u_{0j}x_n)\e_{0i}^{\star} - (x_i+u_{0i}x_n)\e_{0j}^{\star},& \quad 0 \leq j < i \leq n-r-1 \\ 
\Psi_1(\e_{0i}^{j}) & =  (x_j+u_{0j}x_n)\e_{0i}^{\star} - x_i\e_{0j}^{\star},& j < n-r \leq i \leq n-1 \\ 
\Psi_1(\e_{0i}^{j}) & =   x_j\e_{0i}^{\star} - x_i\e_{0j}^{\star},& \quad n-r \leq j < i \leq n-1 \\ 
\Psi_1(\e_{11}^{0}) & =  x_0\e_{11}^{\star} - x_1^q\e_{01}^{\star} - u_{11}x_n^q\e_{00}^{\star}  - \sum_{\omega \in \mathcal{T}_1} u_{2\omega}\omega \e_{01}^{\star} - \sum_{l=0}^{r-1}u_{1,n-r+l}x_1^q\e_{0,n-r+l}^{\star} 
 \\
& \quad \, \, - \sum_{l=0}^{r-1}\sum_{\omega \in \mathcal{T}_1} u_{2\omega}u_{1,n-r+l}\omega \e_{0,n-r+l}^{\star}+ (u_{01}-u_{1n})\Omega  \\
\Psi_1(\e_{1i}^{0}) & =  x_0\e_{1i}^{\star} - x_1^q\e_{0i}^{\star} - \sum_{\omega \in \mathcal{T}_1}u_{2\omega}\omega \e_{0i}^{\star}  - u_{1i}x_n^q\e_{00}^{\star}  + u_{0i}\Omega, & 2 \leq i \leq n-r-1  \\
\Psi_1(\e_{1i}^{1}) & =  x_1\e_{1i}^{\star} - x_i\e_{11}^{\star} + u_{11}x_n^q\e_{0i}^{\star} - u_{1i}x_n^q\e_{01}^{\star} +  \sum_{l=0}^{r}u_{1,n-r+l}x_{n-r+l}\e_{1i}^{\star}  \\ 
& \quad \, \, - \sum_{l=0}^{r-1} u_{1i}u_{1,n-r+l}x_n\e_{0,n-r+l}^{\star},& 2 \leq i \leq n-r-1 \\
\Psi_1(\e_{1i}^{j}) & = x_j\e_{1i}^{\star} - x_i\e_{1j}^{\star} + u_{1j}x_n^q\e_{0i}^{\star} - u_{1i}x_n^q\e_{0j}^{\star},& 2 \leq j < i \leq n-r-1.
\end{align*}
It will be notationally convenient to separate the cases $q>1$ and $q=1$. If $q > 1$, composing $\Psi_0$ and $\Psi_1$ we obtain
\begin{eqnarray} 
\Psi_0\Psi_1(\e_{0i}^j) & = & 0, \quad 0 \leq j < i \leq n-1\\
\Psi_0(\Psi_1(\e_{1i}^j)) & = & (u_{0i}u_{1j}-u_{0j}u_{1i})x_0x_n^{q+1}, \quad 2 \leq j < i \leq n-r-1 \nonumber\\
\label{deformm1} \Psi_0(\Psi_1(\e_{1i}^0)) & = & (u_{0i}(-u_{2\eta} +\alpha)-u_{00}u_{1i})x_0x_n^{q+1}, \quad 2 \leq i \leq n-r-1\\
\Psi_0(\Psi_1(\e_{11}^0)) & = & ((-u_{2\eta} +\alpha)(u_{01}-u_{1n})-u_{00}u_{11})x_0x_n^{q+1} \nonumber \\
\Psi_0(\Psi_1(\e_{1i}^1)) & =  & (u_{11}u_{0i} - u_{1i}(u_{01}-u_{1n}))x_0x_n^{q+1}, \quad 2 \leq i \leq n-r-1 \nonumber
\end{eqnarray}
with $\alpha =(-1)^{q-1}u_{01}^q + (-1)^{q-1}\sum_{\omega \in \mathcal{V}}u_{2\omega}u_{\omega_0}$.

To compute the obstruction space we just repeat the above computation $\bmod \, \mathfrak{m}_{\mathbf{U}}^{l+1}$. Indeed, for $l\geq 1$ let $\Psi_0^l = \Psi_0 \bmod \mathfrak{m}_{\mathbf{U}}^{l+1}$ and $\Psi_1^l= \Psi_1 \bmod \mathfrak{m}_{\mathbf{U}}^{l+1}$. Then the image of $\Psi_0^l\Psi_1^l$ in $T^2(R/\kk,R)_0 \otimes \mathbf{U}/\mathfrak{m}_{\mathbf{U}}^{l+2}$ is
\begin{eqnarray*}
\Psi_0^l\Psi_1^l(\e_{0i}^j) & \equiv & 0, \quad 0 \leq j < i \leq n-1 \\
\Psi_0^l(\Psi_1^l(\e_{1i}^j)) & \equiv & (u_{0i}u_{1j}-u_{0j}u_{1i})x_0x_n^{q+1}, \quad 2 \leq j < i \leq n-r-1\\
\Psi_0^l(\Psi_1^l(\e_{1i}^0)) & \equiv & (u_{0i}(-u_{2\eta} +\alpha)-u_{00}u_{1i})x_0x_n^{q+1}, \quad 2 \leq i \leq n-r-1\\
\Psi_0^l(\Psi_1^l(\e_{11}^0)) & \equiv & ((-u_{2\eta} +\alpha)(u_{01}-u_{1n})-u_{00}u_{11})x_0x_n^{q+1} \\
\Psi_0^l(\Psi_1^l(\e_{1i}^1)) & \equiv  & (u_{11}u_{0i} - u_{1i}(u_{01}-u_{1n}))x_0x_n^{q+1}, \quad 2 \leq i \leq n-r-1
\end{eqnarray*} 

Using \autoref{T2} \textit{i)}, the above equation allows us to directly read off the obstruction to lift our family from the $(l-1)$-th order to $l$-th order (beginning with $l=1$). In particular, the ideal of obstructions to lift to $q$-th order is the $2\times 2$ minors of
\begin{equation*}
\begin{pmatrix}
u_{00} & u_{01} - u_{1n} & u_{02} & u_{03} & \cdots & u_{0,n-r-1} \\
-u_{2\eta} + \alpha & u_{11} & u_{12} & u_{13} & \cdots & u_{1,n-r-1}\\
\end{pmatrix}.
\end{equation*}
If we denote this ideal by $J$, we have $\Psi_0\Psi_1 = 0$ in  $\mathbf{U}/J$ (Equation \autoref{deformm1}). Thus, $\Psi_0$ gives a versal deformation of $I(\lambda)$. Since we are working analytically, we may apply the isomorphism 
that maps $u_{2\eta} \mapsto -u_{2\eta} + \alpha$
and fixes the other variables. This transformation makes $J$ the $2 \times 2$ minors of a generic matrix. Finally, adding back the trivial deformations we obtain the universal deformation space of $I(\lambda)$. 

If $q=1$ we obtain
\begin{eqnarray*}
\Psi_0\Psi_1(\e_{0i}^j) & = & 0, \quad 0 \leq j < i \leq n-1\\
\Psi_0(\Psi_1(\e_{1i}^j)) & = & (u_{0i}u_{1j}-u_{0j}u_{1i})x_0x_n^2, \quad 2 \leq j < i \leq n-r-1\\
\Psi_0(\Psi_1(\e_{1i}^0)) & = & (u_{0i}u_{01}-u_{00}u_{1i})x_0x_n^2 , \quad 2 \leq i \leq n-r-1 \\
\Psi_0(\Psi_1(\e_{11}^0)) & = & (u_{01}(u_{01}-u_{1n})-u_{00}u_{11})x_0x_n^2\\
\Psi_0(\Psi_1(\e_{1i}^1)) & =  & (u_{11}u_{0i} - u_{1i}(u_{01}-u_{1n}))x_0x_n^2, \quad 2 \leq i \leq n-r-1.
\end{eqnarray*}
Arguing as in the $q>1$ case we see that the versal deformation space is cut out by $2\times2$ minors of
\begin{equation*}
\begin{pmatrix}
u_{00} & u_{01} - u_{1n} & u_{02} & u_{03} & \cdots & u_{0,n-r-1} \\
u_{01} & u_{11} & u_{12} & u_{13} & \cdots & u_{1,n-r-1}\\
\end{pmatrix}.
\end{equation*}

We have obtained the desired \'etale-local description as the Segre embedding $\mathbf{P}^1 \times \mathbf{P}^{n-r-1} \hookrightarrow \PP^{2(n-r)-1}$ is cut out by the ideal of $2\times 2$ minors of a generic $2 \times (n-r)$ matrix. It is well known that the Segre embedding is normal and Cohen-Macaulay \cite{HOCHSTER_EAGON}. It follows that the Hilbert scheme is normal and Cohen-Macaulay in a neighbourhood of $[I(\lambda)]$. Combining this with \cite[Theorem 1.4]{rs} and \autoref{borel2} we deduce that the Hilbert scheme is normal and Cohen-Macalay. Since the Hilbert scheme is connected \cite[Corollary 5.9]{HARTSHORNE_CONNECTED}, it must be irreducible. Finally, the description of the general member is given in \cite[Theorem 4.1]{rs} and the other statements follow from \autoref{borel2}.
\end{proof}

\begin{proof}[Proof of \autoref{THEOREM_A} (vii)] Let $\mathbf{U} =\kk[\![u_{00},\dots,u_{0,n-1},u_{11},\dots,u_{1,n-1},v_{00},\dots,v_{0,n-1},v_{11}]\!]$. For convenience we will sometimes use $u_{10}$ to denote $u_{01}$. Consider the following perturbation of $\psi_0$
\begin{eqnarray*}
\Psi_0(\e_{0i}^{\star}) & = &  x_0x_i +u_{0i}x_0x_n + v_{0i}x_{1}x_n, \quad 0 \leq i \leq n-1 \\
\Psi_0(\e_{11}^{\star}) &= &  x_1^2 +u_{11}x_0x_n + v_{11}x_1x_n\\
\Psi_0(\e_{1i}^{\star}) & = &  x_1x_i +u_{1i}x_0x_n, \quad 2 \leq i \leq n-1\\
\end{eqnarray*}
and a perturbation of $\psi_1$ 
\begin{equation*}
\begin{aligned}
\Psi_1(\e_{0i}^{0}) & = (x_0+u_{00}x_n)\e_{0i}^{\star} - (x_i+u_{0i}x_n)\e_{00}^{\star} + v_{00}x_n\e_{1i}^{\star} - v_{0i}x_n\e_{01}^{\star},& 1 \leq i \leq n-1 \\ 
\Psi_1(\e_{0i}^{j}) & = (x_j+u_{0j}x_n)\e_{0i}^{\star} - (x_i+u_{0i}x_n)\e_{0j}^{\star} + v_{0j}x_n\e_{1i}^{\star} - v_{0i}x_n\e_{1j}^{\star},& 1 \leq j <i \leq n-1 \\ 
\Psi_1(\e_{11}^{0}) & = (x_0 + v_{01}x_n)\e_{11}^{\star} - x_1\e_{01}^{\star} - u_{11}x_n\e_{00}^{\star} + (u_{01} - v_{11})x_{n}\e_{01}^{\star}	 \\
\Psi_1(\e_{1i}^{0}) & = x_0\e_{1i}^{\star} - x_1\e_{0i}^{\star} + v_{0i}x_n\e_{11}^{\star}+ u_{0i}x_n\e_{01}^{\star} - u_{1i}x_n\e_{00}^{\star},& \quad 2 \leq i \leq n-1 \\
\Psi_1(\e_{1i}^{1}) & = (x_1 + v_{11}x_n)\e_{1i}^{\star}  - x_i\e_{11}^{\star}+ u_{11}x_n\e_{0i}^{\star} - u_{1i}x_n\e_{01}^{\star},& \quad 2 \leq i \leq n-1 \\
\Psi_1(\e_{1i}^{j}) & = x_j\e_{1i}^{\star} - x_i\e_{1j}^{\star} + u_{1j}x_n\e_{0i}^{\star} - u_{1i}x_n\e_{0j}^{\star},& \quad 2 \leq j < i \leq n-1.
\end{aligned}
\end{equation*}
Composing the two we obtain 
\begin{equation*}
\begin{aligned}
\Psi_0\Psi_1(\e_{01}^0) & = (u_{11}v_{00} - u_{01}v_{01})x_0x_n^2 + (v_{01}(u_{00}-v_{01})- v_{00}(u_{01}-v_{11}))x_1x_n^2 \\
\Psi_0\Psi_1(\e_{0i}^0) & = (u_{1i}v_{00} - u_{01}v_{0i})x_0x_n^2 + (v_{0i}(u_{00}-v_{01})- u_{0i}v_{00})x_1x_n^2, & 2 \leq i \leq n-1 \\
\Psi_0\Psi_1(\e_{0i}^1) & = (u_{1i}v_{01}-u_{11}v_{0i})x_0x_n^2 + (v_{0i}(u_{01}-v_{11})-u_{0i}v_{01})x_1x_n^2,& 2 \leq i < n \\
\Psi_0\Psi_1(\e_{0i}^j) & = (u_{1i}v_{0j}-u_{1j}v_{0i})x_0x_n^2 + (u_{0j}v_{0i}-u_{0i}v_{0j})x_1x_n^2,& 2 \leq j < i < n \\
\Psi_0(\Psi_1(\e_{11}^0)) & =  (u_{01}(u_{01}-v_{11})- u_{11}(u_{00}-v_{01}))x_0x_n^2 + (u_{01}v_{01}-u_{11}v_{00})x_1x_n^2 \\
\Psi_0(\Psi_1(\e_{1i}^0)) & = (\underline{u_{11}v_{0i}+u_{01}u_{0i}-u_{1i}u_{00}})x_0x_n^2 + (\underline{u_{0i}v_{01}+v_{11}v_{0i}-u_{1i}v_{00}})x_1x_n^2,&  2  \leq  i \leq n-1 \\
\Psi_0(\Psi_1(\e_{1i}^1)) & = (u_{0i}u_{11} - u_{1i}(u_{01}-v_{11}))x_0x_n^2 +(u_{11}v_{0i}-u_{1i}v_{01})x_1x_n^2,& 2 \leq i \leq n-1 \\
%
\Psi_0(\Psi_1(\e_{1i}^j)) & = (u_{ij}u_{0i}-u_{1i}u_{0j})x_0x_n^2 + (u_{ij}v_{0i}-u_{1i}v_{0j})x_1x_n^2,& 2 \leq j < i \leq n-1.
\end{aligned}
\end{equation*}
Since the lifts $\Psi_0$ and $\Psi_1$ are first order, 
we see that the ideal of obstructions to lift to second order is the $2\times 2$ minors of 
\begin{equation*}
\begin{pmatrix}
u_{01} 		&u_{11}			& u_{12} & \cdots & u_{1,n-1}\\
v_{00} 		&v_{01}			& v_{02} & \cdots & v_{1,n-1}\\
u_{00}-v_{01} 	&u_{01} - v_{11} 	& u_{02} & \cdots & u_{0,n-1}\\
\end{pmatrix}.
\end{equation*}
Indeed, most of the minors show up as coefficients of $x_0x_n^2$ and $x_1x_n^2$. The other minors come from the underlined equations
\begin{eqnarray*}
u_{11}v_{0i}+u_{01}u_{0i}-u_{1i}u_{00} + (u_{1i}v_{00} - u_{01}v_{0i})& =& v_{01}u_{0i}-v_{0i}(u_{01}-v_{11})\\
u_{0i}v_{01}+v_{11}v_{0i}-u_{1i}v_{00} - (u_{11}v_{0i}-u_{1i}v_{01})& =& u_{01}u_{0i}-u_{1i}(u_{00}-v_{01}).
\end{eqnarray*} 
If we denote the ideal of $2\times 2$ minors by $J$ we have $\Psi_0\Psi_1 = 0$ in  $\mathbf{U}/J$. Thus, $\Psi_0$ gives a versal deformation of $I(\lambda)$. Adding back the trivial deformations gives us the universal deformation space of $I(\lambda)$. This gives us the desired \'etale-local description as the Segre embedding $\mathbf{P}^2 \times \mathbf{P}^{n-1} \hookrightarrow \PP^{3n-1}$ is cut out by the ideal of $2\times 2$ minors of a generic $3 \times n$ matrix. Similar to the previous proof, the other statements follow from \cite{HOCHSTER_EAGON}, \cite[Corollary 5.9]{HARTSHORNE_CONNECTED}, \autoref{borel} and \cite[Theorem 4.1]{rs}. 
\end{proof}

\begin{proof}[Proof of \autoref{THEOREM_A} (ix)] Let $\mathbf{U} = \kk[\![u_{00},\dots,u_{0,n-1},u_{11},\dots,u_{1n}]\!]$ and let $\mathfrak{m}_{\mathbf{U}}$ denote its maximal ideal. We will sometimes use $\e_{10}^{\star}$ to denote $\e_{01}^{\star}$. This does not cause any confusion as $\e_{10}^{\star}$ is not part of a basis of $F_0$. Consider the following perturbation of $\psi_0$ 
\begin{align*}
\Psi_0(\e_{00}^{\star})&=   x_0^2 + u_{00}x_0x_n \\
\Psi_0(\e_{01}^{\star}) & =   x_0x_1 +u_{01}x_0x_n - u_{0,n-1}u_{1,n-1}x_{1}x_n \\
\Psi_0(\e_{0i}^{\star}) & =   x_0x_i +u_{0i}x_0x_n,&  2 \leq i \leq n-2 \\
\Psi_0(\e_{0,n-1}^{\star}) &= x_0x_{n-1} + u_{0,n-1}x_1x_n \\
\Psi_0(\e_{11}^{\star}) &=   x_1^2 +u_{11}x_0x_n + u_{1,n-1}x_1x_{n-1} + u_{1n}x_1x_{n}  \\
\Psi_0(\e_{1i}^{\star}) &= x_1x_i+u_{1i}x_0x_n,& 2 \leq i \leq n-2.
\end{align*}
and a perturbation of $\psi_1$ 
\begin{align*}
\Psi_1(\e_{01}^{0})&= (x_0+ u_{00}x_n)\e_{01}^{\star} - (x_1  + u_{01}x_n)\e_{00}^{\star} +u_{0,n-1}u_{1,n-1}x_n\e_{01}^{\star}&\\
\Psi_1(\e_{0i}^{0})& = (x_0 + u_{00}x_n)\e_{0i}^{\star}- (x_i + u_{0i}x_n)\e_{00}^{\star},&  2 \leq i \leq n-2  \\
\Psi_1(\e_{0i}^{1})& = (x_1+u_{01}x_n)\e_{0i}^{\star} - (x_i+u_{0i}x_n)\e_{01}^{\star}  - u_{0,n-1}u_{1,n-1}x_n\e_{1i}^{\star},&  2 \leq  i \leq n-1\\
\Psi_1(\e_{0i}^{j})& = (x_j+u_{0j}x_n)\e_{0i}^{\star} - (x_i+u_{0i}x_n)\e_{0j}^{\star},&  2 \leq j < i \leq n-2 \\
\Psi_1(\e_{0,n-1}^{j})& = (x_j+u_{0j}x_n)\e_{0,n-1}^{\star} - x_{n-1}\e_{0j}^{\star} -u_{0,n-1}x_n\e_{1j}^{\star},&  0 \leq j \leq n-2\\ 
\Psi_1(\e_{11}^{0})& = x_0\e_{11}^{\star} - x_1\e_{01}^{\star} - u_{11}x_n\e_{00}^{\star} -u_{1,n-1}x_1\e_{0,n-1}^{\star}+ (u_{01}-u_{1n})x_n\e_{01}^{\star} \\
\Psi_1(\e_{1i}^{0})& = x_0\e_{1i}^{\star} - x_1\e_{0i}^{\star} + u_{0i}x_n\e_{01}^{\star} - u_{1i}x_n\e_{00}^{\star},&  2 \leq i \leq n-2\\
\Psi_1(\e_{1i}^{1})& = x_1\e_{1i}^{\star} - x_i\e_{11}^{\star} + u_{11}x_n\e_{0i}^{\star} - u_{1i}x_n\e_{01}^{\star} \\
&+ (u_{1,n-1}x_{n-1}+u_{1n}x_n)\e_{1i}^{\star} - u_{1i}u_{1,n-1}x_n\e_{0,n-1}^{\star},&  2 \leq i \leq n-2 \\
\Psi_1(\e_{1i}^{j})& = x_j\e_{1i}^{\star} - x_i\e_{1j}^{\star} + u_{1j}x_n\e_{0i}^{\star} - u_{1i}x_n\e_{0j}^{\star},& 2 \leq j < i \leq n-2.
\end{align*}

For $l \geq 1$ let,  $\Psi^l_0 \equiv \Psi^l_0 \bmod \mathfrak{m}_{\mathbf{U}}^{l+1}$ and $\Psi^l_1 \equiv \Psi^l_1 \bmod \mathfrak{m}_{\mathbf{U}}^{l+1}$. As done previously, the obstruction to lifting to second order is the image of $\Psi_0^{1}\Psi_{1}^1$ in $T^2(R/\kk,R)_0 \otimes \mathfrak{m}_{\mathbf{U}}^2/\mathfrak{m}_{\mathbf{U}}^3$. This is
\begin{align*}
\Psi_0^1\Psi_1^1(\e_{0i}^j) &\equiv  0 ,& \quad 0 \leq j < i \leq n-2\\
\Psi_0^1\Psi_1^1(\e_{0,n-1}^0)  & \equiv  - u_{01}u_{0,n-1}x_0x_n^2 + u_{00}u_{0,n-1}x_1x_n^2  \\
\Psi_0^1\Psi_1^1(\e_{0,n-1}^{1})  & \equiv   u_{0,n-1}(u_{01}-u_{1n})x_1x_n^2 - u_{0,n-1}u_{11}x_0x_n^2 \\ 
& \quad -\underline{u_{0,n-1}u_{1,n-1} x_1x_{n-1}x_n}\\
\Psi_0^1\Psi_1^1(\e_{0,n-1}^{j})  &\equiv   u_{0j}u_{0,n-1}x_1x_n^2 - u_{0,n-1}u_{1j}x_0x_n^2,& \quad 2 \leq j \leq n-2 \\
\Psi_0^1\Psi_1^1(\e_{11}^0)  &\equiv   (u_{01}(u_{01}-u_{1n})-u_{00}u_{11})x_0x_n^2 \\
&\quad - \underline{u_{0,n-1}u_{1,n-1}x_1^2x_n} \\
\Psi_0^1\Psi_1^1(\e_{1i}^0)  &\equiv   (u_{01}u_{0i}-u_{00}u_{1i})x_0x_n^2,& \quad 2 \leq i \leq n-2 \\
\Psi_0^1\Psi_1^1(\e_{1i}^1)  & \equiv  (u_{0i}u_{11}-u_{1i}(u_{01}-u_{1n}))x_0x_n^2 \\
& \quad + \underline{u_{1,n-1}u_{1i}x_0x_{n-1}x_n}, &\quad 2 \leq i \leq n-2 \\
\Psi_0^1\Psi_1^1(\e_{1i}^{j})  & \equiv   (u_{0i}u_{1j}- u_{0j}u_{1i})x_0x_n^2, &\quad 2 \leq j < i \leq n-2.
\end{align*}
In this image, the three underlined terms are $0$. Indeed, the second and third underlined term (from the top) are $0$ in $R$ and the first term is equal to $\overline{\psi_1^{\vee}}(u_{0,n-1}u_{1,n-1}x_1x_n\f_{01}^{\star})$. After the underlined terms vanish, $\Psi_0^1\Psi_1^1$ is written in terms of our desired basis elements (\autoref{T2} \textit{ii)}). Thus, the ideal generated by the coefficients, which we denote by $J_1$, is the ideal of of obstructions to lift to second order. Let $\mathbf{U}^1 = \mathbf{U}/J_1$ and $\mathfrak{m}_{{\mathbf{U}}^1}$ its maximal ideal. To compute the the obstructions to third order we compute $\Psi^2_0\Psi^2_1$ in $T^2(R/\kk,R)_0 \otimes \mathfrak{m}_{\mathbf{U}^1}^3/\mathfrak{m}_{\mathbf{U}^1}^4$. This is 
\begin{align*}
\Psi_0^2\Psi_1^2(\e_{0i}^j) &\equiv 0, \quad (i,j) \ne (0,n-1) \\
\Psi_0^2\Psi_1^2(\e_{0,n-1}^0)  &\equiv   u_{0,n-1}^2u_{1,n-1}x_1x_n^2  \\
\Psi_0^2\Psi_1^2(\e_{1i}^j)  &\equiv  0, \quad \text{ for all } j,i \\
\end{align*}
Thus, the ideal of obstructions to lift to third order is
\begin{align*}
J_2 & := \left((u_{0,n-1}) + I_{2}
\begin{pmatrix}
u_{00} & u_{01} - u_{1n} & u_{02} & u_{03} & \cdots & u_{0,n-2} \\
u_{01} & u_{11} &  u_{12} & u_{13} & \cdots & u_{1,n-2}\\
\end{pmatrix}
\right)
\cap  \\
& \qquad (u_{00}+u_{0,n-1}u_{1,n-1},u_{01},u_{02},\dots,u_{0,n-2}, u_{11},u_{12},\dots,u_{1,n-2},u_{1n}).
\end{align*}
Here $I_2(-)$ denotes the ideal of the $2\times 2$ minors of $-$. Finally, it is easy to see that $\Psi_0\Psi_1 = 0$ in $\mathbf{U}/J_2$ (for instance, the underlined terms in $\Psi^1_0\Psi^1_1$ are cancelled by the second order terms). Thus $\Psi_0$ gives a versal deformation of $I(\lambda)$. Adding back the trivial deformations gives us the universal deformation space of $I(\lambda)$.

From \autoref{tangent2} and \autoref{nontrivial2} we see that there are $4n-6$ trivial deformations; denote them by $t_1,\dots,t_{4n-6}$. Thus, the smooth component of $\Spec(\mathbf{U}[t_{1},\dots,t_{4n-6}]/J_2)$ has dimension $4n-4$. Since $P_\lambda = \binom{t+n-2}{n-2}+t+1$, there is an irreducible component, $\mathcal{Y}_1$, whose general member parameterizes a line and a disjoint $(n-2)$-plane. This is birational to $\mathbf{G}(1,n) \times \mathbf{G}(1,n-2)$ and, as a consequence, has dimension $4n-4$; thus $\mathcal{Y}_1$ is the smooth component. It is shown in \cite[Theorem B]{ritvik} that  $\mathcal{Y}_1$ is isomorphic to a blow up of $\mathbf{G}(1,n) \times \mathbf{G}(n-2,n)$ along the locus $\{(L,\Lambda): L \subseteq \Lambda \}$. Similar to the previous proofs, the other statements follow from \cite{HOCHSTER_EAGON}, \cite[Corollary 5.9]{HARTSHORNE_CONNECTED}, \autoref{borel} and \cite[Theorem 4.1]{rs}. 
\end{proof}


\section{Hilbert schemes with three Borel-fixed points} \label{threeborelpoints}
The goal of this section is to collect various Hilbert schemes with three Borel-fixed points that have appeared in the literature. In contrast with \autoref{THEOREM_A}, we show that these Hilbert schemes can have three irreducible components and that the components can meet each other in different ways; see \autoref{THREE_1} and  \autoref{THREE_2}. For simplicity, we assume $\text{char}(\kk) = 0$. 

\subsection{Plane curves and two isolated points} Let $\lambda = (2^q,1,1)$ with $q \geq 5$. The Hilbert scheme $\Hilb^{P_\lambda}(\PP^3)$ is irreducible and singular with general member parameterizing a plane curve of degree $q$ union $2$ isolated \break points \cite[Theorem 4.2, Remark 4.4]{cn}. Using \autoref{moo} one can check that its Borel-fixed points are 
$(x_0) +x_1^{d}(x_1,x_2^2)$,
$x_0(x_0,x_1,x_2) + x_1^{q}(x_1,x_2)$,
and
$x_0(x_0,x_1,x_2^2) + (x_1^q)$.

\subsection{Linear space and three isolated points} 
For $\lambda = (n-1,1,1,1)$ the Hilbert scheme $\Hilb^{P_\lambda}(\PP^n)$ is irreducible and singular with general point parameterizing an $(n-2)$-plane union $3$ isolated points \cite[Theorem 3.9]{cn}. Using \autoref{moo} one can check that it its Borel-fixed points are
$
(x_0) + x_1(x_1,\dots,x_{n-2},x_{n-1}^3)$,
$(x_0) + x_1(x_1,\dots,x_{n-3}) +x_1(x_{n-2}^2,x_{n-2}x_{n-1},x_{n-1}^2)$,
and
$x_0(x_0,\dots,x_{n-1}) + x_1(x_1,\dots,x_{n-2},x_{n-1}^2).$

\subsection{Five points in $\PP^2$} The Hilbert scheme $\Hilb^{5}(\PP^2)$ is smooth and has three Borel-fixed points given by
$(x_0,x_1^5)$,
$(x_0^2,x_0x_1,x_1^4)$,
and 
$(x_0^2,x_0x_1^2,x_1^3)$.

\subsection{Four points in $\PP^3$} The Hilbert scheme $\Hilb^{4}(\PP^3)$ is singular and Gorenstein \cite{sk}, and has three Borel-fixed points given by
$(x_0,x_1,x_2^4)$,
$(x_0,x_1x_2,x_1^2,x_2^3)$,
and $ (x_0,x_1,x_2)^2$.

\subsection{Conic union two isolated points} \label{THREE_1} Let $\lambda =(2,2,1,1)$. It is shown in \cite[Example 4.6c]{cn} that the Hilbert scheme $\Hilb^{P_\lambda}(\PP^3)$ has three irreducible components. They are 
\begin{itemize}
\item $\mathcal{Y}_1$, with general point parameterizing doubled lines of genus $-2$ with no embedded points,
\item $\mathcal{Y}_2$, with general point parameterizing two skew lines union an isolated point,
\item $\mathcal{Y}_3$, with general point parameterizing a conic union two isolated points.
\end{itemize}
Using \autoref{moo} we obtain
$(x_0,x_1^3,x_1^2x_2^2)$,
$(x_0^2,x_0x_1,x_0x_2,x_1^2x_2,x_1^3)$,
and 
$I:= (x_0^2,x_0x_1,x_1^2,x_0x_2^2)$ 
as the Borel-fixed ideals

We will now show that all three components contain $I$. For any $t \in \mathbf{A}^1 - 0$ the ideal $I_t = (x_0^2,x_0x_1,x_1^2, \break x_0x_2^2-tx_1x_3^2)$ lies in $\mathcal{Y}_1$ \cite[Proposition 1.4]{TRIPLE_STRUCTURES}. Thus the limit, $\lim_{t \to 0} I_t = I$, also lies in $\mathcal{Y}_1$. 
By \cite[Theorem 1.1]{ccn}, the ideal $J = (x_0^2,x_0x_1,x_1^2,x_0x_2)$ lies in the intersection of the component parameterizing two skew lines and the component parameterizing a plane conic union an isolated point. It follows that for $t \ne 0$ the ideal
\begin{equation*}
J_t := (x_0^2-x_0x_2,x_0x_1,x_1^2,x_0x_2^2+tx_0x_2x_3) = (x_0^2,x_0x_1,x_1^2,x_0x_2) \cap (x_1,x_2+tx_3,x_0-x_2).
\end{equation*}
lies  in $\mathcal{Y}_1 \cap \mathcal{Y}_2$.Thus, the limit $(x_0^2-x_0x_2,x_0x_1,x_1^2,x_0x_2^2)$ is in $\mathcal{Y}_1 \cap \mathcal{Y}_2$. Finally, considering the limit of $(x_0^2-tx_0x_2,x_0x_1,x_1^2,x_0x_2^2)$ we see that $[I] \in \mathcal{Y}_1 \cap \mathcal{Y}_2$. 

\subsection{Quadric $d$-fold, a line and a point} \label{THREE_2} Let $\lambda = (d+1,d+1,2,1)$ and $d \geq 2$. The Hilbert scheme $\Hilb^{P_\lambda}(\PP^n)$ has three Borel-fixed points and they are given by
\begin{align*}
& (x_0,\dots,x_{n-d-3}) + x_{n-d-2}(x_{n-d-2},\dots,x_{n-1}) + x_{n-d-1}^2(x_{n-d-1},\dots,x_{n-2}),  \\
&(x_0,\dots,x_{n-d-2}) + x_{n-d-1}^{2}(x_{n-d-1},\dots,x_{n-3}) + x_{n-d-1}^2x_{n-2}(x_{n-2},x_{n-1}), \text{ and }  \\
& (x_0,\dots,x_{n-d-3}) + x_{n-d-2}(x_{n-d-2},\dots,x_{n-2}) + (x_{n-d-1}^2).
\end{align*}
When $d=2$ this Hilbert scheme has three irreducible components, denoted by $H_4,H_4'$ and $H_4''$, such that $H_4 \cap H_4'' = \emptyset$ and $H_4' \cap H_4, H_4' \cap H_4'' \ne \emptyset$ \cite[Proposition 2.7]{ccn}. 


\begin{acknowledgement} I would like to thank Michael Christianson, David Eisenbud and Alessio Sammartano for helpful conversations. I would also like to thank the referees for a thorough reading and suggesting numerous improvements; especially those that simplified Section 2.
\end{acknowledgement}

\end{document}